\documentclass[11pt]{amsart}
\usepackage{amssymb, amsmath, amsfonts, amsthm,hyperref,graphics}
\usepackage[all]{hypcap}
\usepackage[hmargin=1 in, vmargin = 1 in]{geometry}
\usepackage[enableskew,vcentermath]{youngtab} 
\usepackage{tikz}
\usepackage[all]{xy}
\usepackage{mathrsfs}

\usetikzlibrary{arrows}

\newcommand{\ybox}{{\scalebox{.5}{\yng(1)}}}
\newcommand{\tinybox}{{\scalebox{.3}{\yng(1)}}}
\newcommand{\Mo}[1]{\cM_{0,#1}}
\newcommand{\Mbar}[1]{\overline{\Mo{#1}}}
\newcommand{\SLdot}{\cS(\lambda_\bullet)}

\newcommand{\SLdotBox}{\cS(\lambda_\bullet\ ; \ \ybox_z)}
\newcommand{\sh}{\mathrm{sh}}
\newcommand{\ev}{\mathrm{ev}}
\newcommand{\esh}{\mathrm{esh}}
\newcommand{\rsh}{\mathrm{rsh}}
\newcommand{\rect}{{\scalebox{.3}{\yng(3,3)}}}

\newcommand{\newword}[1]{\textbf{\emph{#1}}}

\newcommand{\into}{\hookrightarrow}

\newcommand{\eset}{\varnothing}

\DeclareMathOperator{\Spec}{Spec}

\renewcommand{\AA}{\mathbb{A}}

\newcommand{\CC}{\mathbb{C}}

\newcommand{\HH}{\mathbb{H}}

\newcommand{\PP}{\mathbb{P}}

\newcommand{\RR}{\mathbb{R}}

\newcommand{\ZZ}{\mathbb{Z}}

\newcommand{\cC}{\mathcal{C}}

\newcommand{\cF}{\mathcal{F}}
\newcommand{\cG}{\mathcal{G}}

\newcommand{\cM}{\mathcal{M}}

\newcommand{\cO}{\mathcal{O}}

\newcommand{\cS}{\mathcal{S}}

\newcommand{\cU}{\mathcal{U}}

\newcommand{\sB}{\mathscr{B}}

\newcommand{\sF}{\mathscr{F}}

\newcommand{\sH}{\mathscr{H}}
\newcommand{\sI}{\mathscr{I}}

\newtheorem{thm}{Theorem}[section]
\newtheorem{lemma}[thm]{Lemma}
\newtheorem{cor}[thm]{Corollary}
\newtheorem{prop}[thm]{Proposition}
\newtheorem{conj}[thm]{Conjecture}

\theoremstyle{definition}
\newtheorem{exa}[thm]{Example}
\newtheorem*{rmk}{Remark}

\numberwithin{figure}{section}
\numberwithin{equation}{section}

\theoremstyle{plain}

\title{One-dimensional Schubert problems with respect to osculating flags}
\author{Jake Levinson}

\begin{document}

\begin{abstract}
We consider Schubert problems with respect to flags osculating the rational normal curve. These problems are of special interest when the osculation points are all real -- in this case, for zero-dimensional Schubert problems, the solutions are ``as real as possible''. Recent work by Speyer has extended the theory to the moduli space $\Mbar{r}$, allowing the points to collide. These give rise to smooth covers of $\Mbar{r}(\RR)$, with structure and monodromy described by Young tableaux and jeu de taquin.


In this paper, we give analogous results on one-dimensional Schubert problems over $\Mbar{r}$. Their (real) geometry turns out to be described by orbits of Sch\"{u}tzenberger promotion and a related operation involving tableau evacuation. Over $\Mo{r}$, our results show that the real points of the solution curves are smooth.

We also find a new identity involving `first-order' K-theoretic Littlewood-Richardson coefficients, 
for which there does not appear to be a known combinatorial proof.
\end{abstract}

\maketitle

\begin{rmk}We review terminology relating to Schubert varieties in Section \ref{sec:grass-schubert}; to stable curves and $\Mbar{r}$ in Section \ref{sec:M0r}; and to tableau combinatorics and dual equivalence in Section \ref{sec:tableau-combinatorics}.
\end{rmk}

\section{Introduction}
\subsection{Osculating flags} Consider the following construction: let $f : \mathbb{P}^1 \to \mathbb{P}^{n-1}$ be the Veronese embedding $t \mapsto (t,t^2,\ldots, t^{n-1})$, and consider the Grassmannian
\[G(k,n) = G(k,H^0(\mathcal{O}_{\mathbb{P}^1}(n-1)))\]
of linear series $V$ of rank $k$ and degree $n-1$ on $\mathbb{P}^1$. At each point $f(p) \in \mathbb{P}^{n-1}$, there is the \emph{osculating flag} $\mathscr{F}(p)$ 
of planes intersecting $f(\mathbb{P}^1)$ at $f(p)$ with the highest possible multiplicity. In this paper, we consider Schubert conditions with respect to such flags.

Let $\rect$ denote the $k \times (n-k)$ rectangular partition. For a partition $\lambda \subseteq \rect$, we denote by $\Omega(\lambda,p)$ the Schubert variety for $\lambda$ with respect to $\mathscr{F}(p)$,
and for a collection of distinct points $p_\bullet = (p_1, \ldots, p_r)$ and partitions $\lambda_\bullet = (\lambda_1, \ldots, \lambda_r)$, we set
\[S(\lambda_\bullet, p_\bullet) = \bigcap_{i=1}^r \Omega(\lambda_i, p_i).\]
Note that the codimension of $\Omega(\lambda,p_i)$ is $|\lambda|$. We call the quantity $\rho(\lambda_\bullet) := k(n-k) - \sum |\lambda_i|$ the \emph{expected dimension} of $S(\lambda_\bullet, p_\bullet)$.

Geometrically, such Schubert conditions describe linear series $V \subset H^0(\mathcal{O}_{\mathbb{P}^1}(n-1))$ satisfying specified \emph{vanishing conditions} at each $p_j$. That is, the finite set
$\{\mathrm{ord}_p(s) : s \in V\} \subset \mathbb{Z}_{\geq 0}$
of orders of vanishing at $p$ of sections $s\in V$ is specified. These conditions first arose in the study of limit linear series \cite{EH86}. There it was shown that a collection of linear series on the components of a reducible nodal curve $C$ occurs as the limit of a single linear series on a smoothing of $C$ if and only if the collection satisfies `complementary' vanishing conditions with respect to the nodes of $C$.

On the other hand, such Schubert conditions have been of interest in intersection theory and real Schubert calculus, thanks to such transversality theorems as:

\begin{thm}\emph{\cite{EH86}} \label{thm:EH86}
For any choice of points $p_i$ and partitions $\lambda_i$, the intersection $S(\lambda_\bullet,p_\bullet)$ is dimensionally transverse. (It is empty if $\rho(\lambda_\bullet) < 0$.)
\end{thm}
and
\begin{thm} \emph{\cite{MTV09}} \label{thm:MTV09}
If the $p_i$ are all in $\mathbb{RP}^1$ and $\rho(\lambda_\bullet) = 0$, then $S(\lambda_\bullet, p_\bullet)$ is reduced and consists entirely of real points.
\end{thm}
Theorem \ref{thm:MTV09} (originally known as the \emph{Shapiro-Shapiro Conjecture}) has inspired work relating the real structure of $S(\lambda_\bullet, p_\bullet)$, as the points $p_\bullet$ vary, to combinatorial Schubert calculus and the theory of Young tableaux. An excellent survey of this material is \cite{Sottile}. The key observation is that the cardinality of $S(\lambda_\bullet, p_\bullet)$ is the Littlewood-Richardson coefficient $c_{\lambda_1, \ldots, \lambda_r}^{\rect}$. We may then ask if there is a canonical bijection between the points of $S$ and the corresponding combinatorial objects.

First, we allow the points to vary. The construction above gives a family $S(\lambda_\bullet)$ over the space $\mathcal{U}_r$ of $r$-tuples of distinct points of $\mathbb{P}^1$ or, working up to automorphism, the moduli space $\Mo{r}$. Speyer \cite{Sp} extended the family to allow the points $p_i$ to collide, working over the compactification $\Mbar{r}$.
\begin{thm}\emph{\cite{Sp}}
There are flat, Cohen-Macaulay families $\mathcal{S}(\lambda_\bullet) \subset \mathcal{G}(k,n) \to \Mbar{r}$, whose restriction to $\Mo{r}$ is $S(\lambda_\bullet) \subset G(k,n) \times \Mo{r} \to \Mo{r}$. The boundary fibers of $\mathcal{G}(k,n)$ consist of limit linear series and the boundary fibers of $\mathcal{S}(\lambda_\bullet)$ consist of limit linear series satisfying the conditions $\lambda_\bullet$ at the marked points.
\end{thm}

In the case of zero-dimensional Schubert problems, the real locus of $\mathcal{S}(\lambda_\bullet)$ has the following remarkable structure:
\begin{thm}\emph{\cite{Sp}} \label{thm:Sp14}
When $\rho(\lambda_\bullet) = 0$, the map of manifolds $\mathcal{S}(\lambda_\bullet)(\mathbb{R}) \to \Mbar{r}(\mathbb{R})$ is a smooth covering map. The fibers of $\mathcal{S}(\lambda_\bullet)(\mathbb{R})$ are indexed by certain collections of Young tableaux; the monodromy of the cover is then given by operations from Sch\"{u}tzenberger's jeu de taquin.
\end{thm}
In addition to giving the desired bijection from points of $S$ to tableaux, Theorem \ref{thm:Sp14} provides a geometrical interpretation of jeu de taquin as the result of lifting arcs from $\Mbar{r}(\RR)$ to $\mathcal{S}(\lambda_\bullet)(\mathbb{R})$, i.e. varying the Schubert problem in real 1-parameter families. Related operations such as promotion and evacuation also have geometrical meanings in $\mathcal{S}(\lambda_\bullet)$, and will play a starring role in the main content of this paper, which is the study of one-dimensional Schubert problems (below).

The most important type of marked stable curve $C$ in this setting consists of $r-2$ components connected in a chain. We will call $C$ a \newword{caterpillar curve} (see Figure \ref{fig:caterpillar-curve}). Such a curve is automatically defined over $\mathbb{R}$. The statement of Theorem \ref{thm:Sp14} is simpler for caterpillar curves. For example, let $\mathcal{S} = \mathcal{S}(\ybox, \ldots, \ybox)$, with $k(n-k)$ copies of $\ybox$, and let $C$ be a caterpillar curve. Let $SYT(\rect)$ be the set of standard Young tableaux of shape $\rect$. Then:

\newtheorem*{thmsp}{Theorem 1.4 (special case)}
\begin{thmsp}
The fiber of $\mathcal{S}$ over $C$ is in bijection with $SYT(\rect)$. If we follow an arc to another caterpillar curve $C'$, the tableau is either unchanged or altered by a Bender-Knuth involution.
\end{thmsp}
Purbhoo in \cite{Pu} has analogous results regarding the real monodromy of the Wronski map $G(k,n) \to \mathrm{Hilb}_{k(n-k)}(\PP^1)$. This map associates to a linear series $V$ its higher ramification locus, as a subscheme of $\PP^1$. Here also, the monodromy (over the locus where the map is unramified) is described in terms of jeu de taquin, yielding a geometrical interpretation of JDT and the Littlewood-Richardson rule. The primary difference is that the Wronski map is not a covering map: the fibers collide over the boundary of the Hilbert scheme.

\subsection{The case of curves; results of this paper} We now study the case $\rho(\lambda_\bullet) = 1$, so that $\mathcal{S}(\lambda_\bullet) \to \Mbar{r}$ is a family of curves. We are interested in both the geometry of the family and a combinatorial description of $\mathcal{S}(\lambda_\bullet)(\mathbb{R})$ as a CW-complex. We state our main geometrical result first:

\begin{thm} \label{thm:main-geom}
There is a finite, flat, surjective morphism $\mathcal{S}(\lambda_\bullet) \to \mathcal{C}$ of varieties over $\Mbar{r}$, where $\mathcal{C} \to \Mbar{r}$ is the universal curve. This map is defined over $\mathbb{R}$, is \'{e}tale over the real points of $\mathcal{C}$, and the preimage of every real point consists entirely of real points. In particular, for $[C] \in \Mbar{r}(\mathbb{R})$, the map of curves $\mathcal{S}(\lambda_\bullet)|_{[C]}(\mathbb{R}) \to C(\mathbb{R})$ is a covering map. 
\end{thm}

The key idea behind Theorem \ref{thm:main-geom} is the following. A point $s \in S(\lambda_\bullet, p_\bullet)$ is a solution to an `underspecified' Schubert problem, and it is not hard to show that, for generic $s$, there is a unique $(r+1)$-st point $z \in \mathbb{P}^1 - \{p_\bullet\}$, such that $s$ satisfies the single-box Schubert condition $\Omega(\ybox, z)$. We show that the assignment $s \mapsto z$ extends to a morphism $S(\lambda_\bullet, p_\bullet) \to \mathbb{P}^1$. We then extend this construction to the boundary fibers. (We think of $z$ as an additional `moving $\ybox$ condition'.) In particular, thinking of $\mathcal{C} \to \Mbar{r}$ as the `forgetting map' $\Mbar{r+1} \to \Mbar{r}$, we have a diagram
\[
\xymatrix{
\cS(\lambda_\bullet;\ \ybox_{r+1}) \ar[r] \ar[d]_{\pi} & \Mbar{r+1} \ar[d] \\
\cS(\lambda_\bullet) \ar[r] \ar@{-->}[ur]^f & \Mbar{r}
}
\]
and we show that $\pi$ is an isomorphism of total spaces. The map $f$ is the map of Theorem \ref{thm:main-geom}. We then use the description of the total space of the zero-dimensional Schubert problem $\cS(\lambda_\bullet;\ \ybox_{r+1})$ to study the (one-dimensional) fibers of $\cS(\lambda_\bullet)$.\\

Over $\Mo{r}$, this result leads to the following:
\newtheorem*{cor:as-real-as-poss}{Corollary \ref{cor:as-real-as-poss}}
\begin{cor:as-real-as-poss}
If the $p_i$ are all in $\mathbb{RP}^1$, the curve $S = S(\lambda_\bullet, p_\bullet) \subset G(k,n)$ has smooth real points. Moreover, $S(\mathbb{C}) - S(\mathbb{R})$ is disconnected.
\end{cor:as-real-as-poss}

We think of Corollary \ref{cor:as-real-as-poss} as saying that $S$ is `almost as real as possible' when all the $p_j$ are real. We say `almost' because while it is often desirable for a real integral algebraic curve of genus $g$ to have $g+1$ real connected components, this is not the case for $S$ (see Example \ref{exa:large-k} for a smooth curve of genus $2$ with one real connected component). Instead, the real connected components of $S$ are determined by combinatorial data, which we state below. Note that we do not assert in general that $S$ is smooth or integral, though it is reduced. In fact there are cases where $\chi(\mathcal{O}_S) > 1$ (see Example \ref{exa:disconnected}), from which we observe:

\newtheorem*{cor*}{Corollary}
\begin{cor*}
A one-dimensional Schubert problem in $G(k,n)$ need not be a connected curve.
\end{cor*}

We remark that our other results primarily concern \emph{fibers} of $\cS(\lambda_\bullet)$, not its total space. The latter is isomorphic to the total space of $\cS(\lambda_\bullet;\ \ybox_{r+1})$, hence has a description from Theorem \ref{thm:Sp14}. We do note that Theorem \ref{thm:main-geom} implies that the topology of the fiber $S(\lambda_\bullet, p_\bullet)(\mathbb{R})$ does not change over a connected component $X \subset \Mo{r}(\mathbb{R})$. In particular:

\begin{cor*}
Each real connected component of $S(\lambda_\bullet)(\mathbb{R})|_X$ is homeomorphic to a cylinder $S^1 \times X$.
\end{cor*}

We now describe the real topology of the fibers of $\mathcal{S}(\lambda_\bullet)(\mathbb{R})$ in terms of Young tableaux. Our description extends that of Theorem \ref{thm:Sp14} via the isomorphism $\pi$ above, and is in terms of orbits of Young tableaux and dual equivalence classes under operations related to Sch\"{u}tzenberger promotion and evacuation.

We define a \newword{chain of dual equivalence classes from $\alpha$ to $\beta$} to be a sequence ${\bf D} = (D_1, \ldots, D_r)$ of dual equivalence classes of skew standard Young tableaux, such that $\sh(D_1)$ extends $\alpha$, $\sh(D_{i+1})$ extends $\sh(D_i)$ for each $i$, and $\beta$ extends $\sh(D_r)$. We say the chain has \newword{type} $(\lambda_1, \ldots, \lambda_r)$ if $\lambda_i$ is the rectification shape of $D_i$. Let $X_\alpha^\beta(\lambda_\bullet) := X_\alpha^\beta(\lambda_1, \ldots, \lambda_r)$ denote the set of such chains. In section \ref{sec:shuffling-ops}, we define noncommuting involutions $\sh_i$ and $\esh_i$, called {\bf shuffling} and {\bf evacuation-shuffling}, both of which switch $\lambda_i$ and $\lambda_{i+1}$ in the type of the chain. We note that $X_\eset^\lambda(\ybox, \ldots, \ybox)$, with $|\lambda|$ copies of $\ybox$, is just $ SYT(\lambda)$, and under this identification $\esh_i$ is the identity function and $\sh_i$ is the $i$-th Bender-Knuth involution. We note that Sch\"{u}tzenberger promotion on $SYT(\lambda)$ then corresponds to the composition
\[\sh_{|\lambda|-1} \circ \cdots \circ \sh_2 \circ \sh_1 : X_\eset^\lambda(\ybox, \ldots, \ybox) \to X_\eset^\lambda(\ybox, \ldots, \ybox).\]
We think of chains of dual equivalence classes as generalizations of standard tableaux. \\

Our main combinatorial result is the following:

\newtheorem*{thm:DE-caterpillar-covering}{Theorem \ref{thm:DE-caterpillar-covering}}
\begin{thm:DE-caterpillar-covering}
Let $C$ be a caterpillar curve with marked points $p_1, \ldots, p_r$ from left to right. Let $S = \mathcal{S}(\lambda_\bullet)|_{[C]}$. The covering map $S(\mathbb{R}) \to C(\mathbb{R})$ is as follows:
\begin{enumerate}
\item[(i)] If $q$ is the node between $p_i$ and $p_{i+1}$, the fiber of $S$ over $q$ is indexed by the set \[X_\eset^{\rect}(\lambda_1, \ldots, \lambda_i,\ybox,\lambda_{i+1}, \ldots, \lambda_r).\] The fibers over $p_1$ and $p_r$ are analogous, with $\ybox$ in the second and second-to-last positions, respectively.
\end{enumerate}
Then, for $i = 2, \ldots, r-1$, we have:
\begin{enumerate}
\item[(ii)] The arc \emph{through $p_i$} lifts to an arc from ${\bf D}$ to $\esh_i({\bf D})$, where
\[\esh_i : X_\eset^{\rect}(\lambda_1, \ldots, \ybox,\lambda_i, \ldots, \lambda_r) \to X_\eset^{\rect}(\lambda_1, \ldots, \lambda_i,\ybox, \ldots, \lambda_r)\] is the $i$-th evacuation-shuffle.

\item[(iii)] The arc \emph{opposite $p_i$} lifts to an arc from ${\bf D}$ to $\sh_i({\bf D})$, where
\[\sh_i : X_\eset^{\rect}(\lambda_1, \ldots, \ybox,\lambda_i, \ldots, \lambda_r) \to X_\eset^{\rect}(\lambda_1, \ldots, \lambda_i,\ybox, \ldots, \lambda_r)\] is the $i$-th shuffle.
\end{enumerate}
\end{thm:DE-caterpillar-covering}
\begin{figure}[h]
\centering
\includegraphics[scale=0.7]{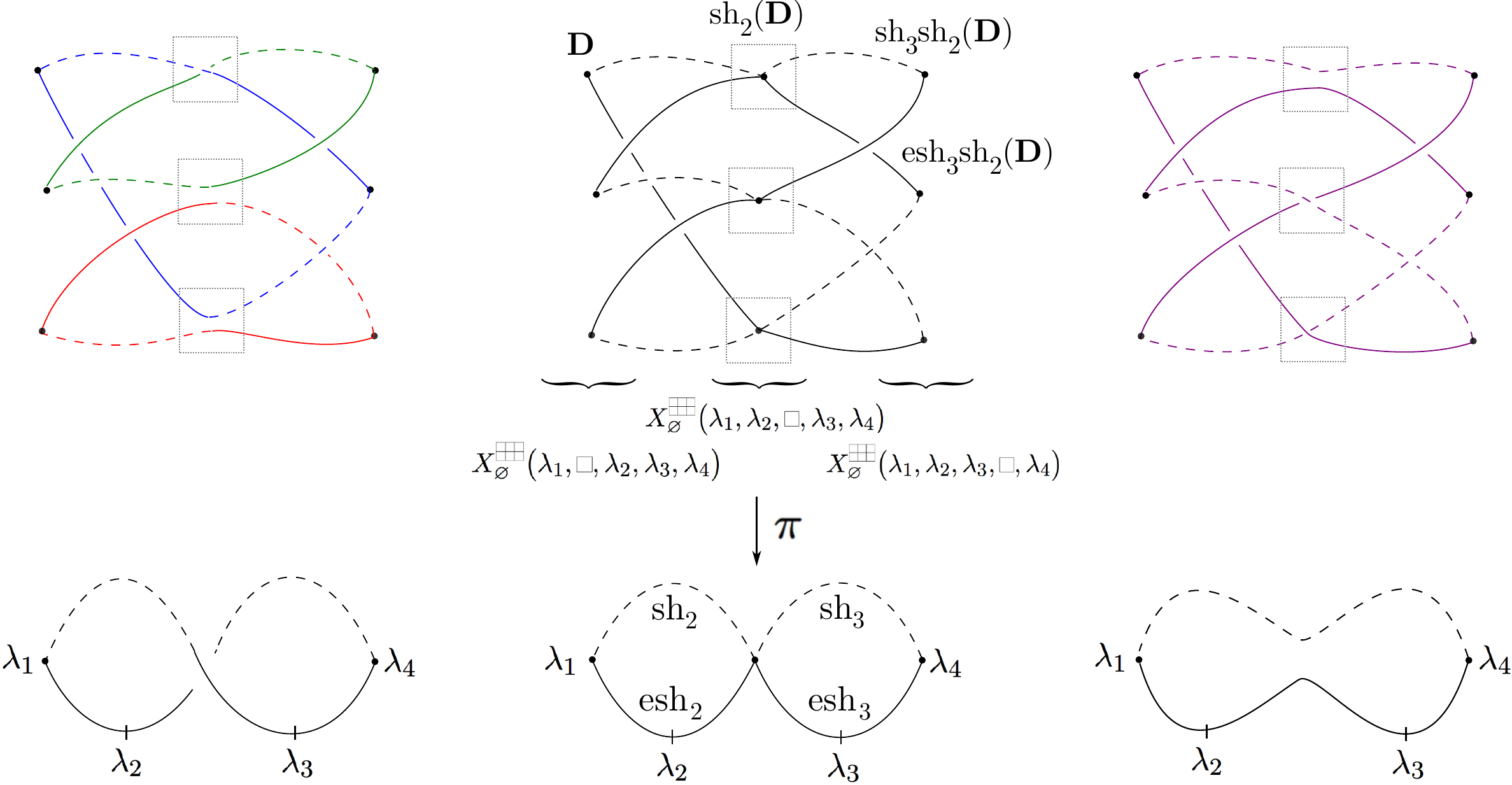} \\
\caption{Center: a covering map $\pi : S(\RR) \to C(\RR)$, where $C$ is a caterpillar curve with marked points $p_1, \ldots, p_4$. Left, right: two nearby desingularizations, obtained by smoothing the node in two different ways. Note that the number of connected components may change.}
\label{fig:covering-space}
\end{figure}

By passing to a nearby desingularization, we obtain a description for fibers over $\Mo{r}$ in terms of orbits of tableau promotion:
\newtheorem*{cor:promotion}{Corollary \ref{cor:promotion}}
\begin{cor:promotion}
If the $p_i$ are all in $\mathbb{RP}^1$ and $S = S(\ybox, \ldots, \ybox \ ; p_\bullet) \subset G(k,n)$, there is a bijection
\[\left\{\begin{split}\text{components\ } \\ \text{ of } S(\mathbb{R}) \hspace{0.4cm} \end{split}\right\} \longleftrightarrow SYT(\rect)/\omega,\]
where $\omega : SYT(\rect) \to SYT(\rect)$ is Sch\"{u}tzenberger promotion.

Similarly, if $S = S(\lambda_\bullet, p_\bullet) \subset G(k,n)$, and the circular ordering of the points is $p_1, \ldots, p_r$, there is a bijection
\[\left\{\begin{split}\text{components\ } \\ \text{ of } S(\mathbb{R}) \hspace{0.4cm} \end{split}\right\} \longleftrightarrow X_{\tinybox}^{\rect}(\lambda_\bullet)/\omega',\]
where $\omega'$ is the composition
\[\omega' = \iota^{-1} \circ \esh_1 \cdots \esh_{r-1} \sh_{r-1} \cdots \sh_{1} \circ \iota\]
and $\iota$ is the natural bijection $X_{\tinybox}^{\rect}(\lambda_\bullet) \to X_\eset^{\rect}(\ybox, \lambda_\bullet).$
\end{cor:promotion}
We emphasize that while our proofs rely crucially on degenerations over $\Mbar{r}$, Corollary \ref{cor:promotion} describes Schubert problems contained in a single Grassmannian.

We also note that the operator $\omega'$ depends on the circular ordering of the points $p_\bullet$. If two circular orderings degenerate to the same caterpillar curve $C$, we may view the operators $\omega'_1, \omega'_2$ as different sequences of shuffles and evacuation shuffles applied to $X_{\tinybox}^{\rect}(\lambda_1, \ldots, \lambda_r)$. See Corollary \ref{cor:connected-components-desing}. In general the orbit structure may differ; see Example \ref{exa:minimal-counterex} for an example in $G(3,8)$. A necessary condition, however, is that certain Littlewood-Richardson numbers be greater than 1:

\newtheorem*{cor:mult-free}{Corollary \ref{cor:mult-free}}
\begin{cor:mult-free}
Suppose the pairwise products $\lambda_i \cdot \lambda_j$ in $H^*(G(k,n))$ are multiplicity-free. Then the operators $\omega$ for different circular orderings are all conjugate. In particular, the number of real connected components of $S(\RR)$ does not depend on the ordering of the $p_\bullet$.
\end{cor:mult-free}
We note that the condition above holds for any Schubert problem on $G(2,n)$, and for any Schubert problem in which every $\lambda_i$ is a rectangular partition.

\subsection{The genus of $S$ and K-theory} A smooth, integral algebraic curve $S$ defined over $\mathbb{R}$ that is disconnected by its real points has the property
\[\#\left\{\begin{split}\text{components\ } \\ \text{ of } S(\mathbb{R}) \hspace{0.4cm} \end{split}\right\} \equiv g(S) + 1 \equiv \chi(\mathcal{O}_S) \ (\mathrm{mod}\ 2).\]
In fact, as long as $S(\mathbb{R})$ is smooth, the above equation holds (with $\chi(\mathcal{O}_S)$) even if $S$ is singular or reducible, since its singularities then occur in complex conjugate pairs.

For our curves $S(\lambda_\bullet, p_\bullet)$, we have described the left-hand side of this identity in terms of objects from $H^*(G(k,n))$; on the other hand, we may compute $\chi(\mathcal{O}_S)$ in the K-theory ring $K(G(k,n))$. Let $[\mathcal{O}_\lambda]$ denote the class of the structure sheaf of the Schubert variety for $\lambda$, 
and let $k_{\lambda_\bullet}^\nu$ be the absolute value of the coefficient of $[\mathcal{O}_\nu]$ in the K-theory product $\prod_i [\mathcal{O}_{\lambda_i}]$. This is zero unless $|\nu| \geq \sum |\lambda_i|$, and if equality holds then $k_{\lambda_\bullet}^\nu = c_{\lambda_\bullet}^\nu$. We have the following:
\newtheorem*{cor:parity-eqn}{Corollary \ref{cor:parity-eqn}}
\begin{cor:parity-eqn}
Let $\alpha, \beta, \gamma$ be partitions with $|\alpha| + |\beta| + |\gamma| = k(n-k) - 1$. Let $\omega' = \esh_2 \circ \sh_2$, where
\[\esh_2, \sh_2 : X_\eset^{\rect}(\alpha, \ybox, \beta, \gamma) \to X_\eset^{\rect}(\alpha, \beta, \ybox, \gamma)\]
are the shuffle and evacuation-shuffle operators. Then
\[\#\mathrm{orbits}(\omega') = c_{\alpha \tinybox \beta \gamma}^{\rect} - k_{\alpha \beta}^{\gamma^c} \ (\mathrm{mod} \ 2) \ \ \text{ and }\ \ \mathrm{sign}(\omega') = k_{\alpha \beta}^{\gamma^c} \ (\mathrm{mod} \ 2),\]
where we think of $\omega'$ as a permutation with sign $0$ or $1$. We also have an inequality
\[ c_{\alpha \tinybox \beta \gamma}^{\rect} \leq k_{\alpha \beta}^{\gamma^c} + \#\mathrm{orbits}(\omega').\]
\end{cor:parity-eqn}
Similar statements hold for products of more than three Schubert classes. Corollary \ref{cor:parity-eqn} has intriguing connections to Thomas and Yong's K-theoretic jeu de taquin for increasing tableaux:
\begin{thm}\emph{\cite{ThYo}} \label{thm:thomas-yong-kjdt}
Let $\alpha, \beta, \gamma$ be partitions satisfying $|\alpha| + |\beta| \leq |\gamma^c|$. Then $k_{\alpha \beta}^{\gamma^c}$ is the number of increasing tableaux of shape $\gamma^c/\alpha$ that rectify to the highest-weight tableau of shape $\beta$ under K-theoretic jeu de taquin.
\end{thm}
When $|\alpha| + |\beta| + |\gamma| = k(n-k) - 1$, any such tableau is standard except for a single repeated entry. An element of $X_\eset^{\rect}(\alpha, \ybox, \beta, \gamma)$ is represented by similar data: a filling of $\gamma^c/\alpha$ by, first, a single box extending $\alpha$, say to $\alpha^+$, followed by a standard tableau $T$ of shape $\gamma^c/\alpha^+$, rectifying to $\beta$. The operator $\omega'$ slides the $\ybox$ through $T$; if we view $\ybox$ as an `extra' entry for $T$, we obtain a sequence of increasing tableaux.

Despite this similarity, we do not know a direct combinatorial proof of Corollary \ref{cor:parity-eqn} in general. We do obtain an explicit description of $\omega'$ when $\beta$ is a horizontal or vertical strip (the `Pieri case'):
\newtheorem*{thm:pieri-case}{Theorem \ref{thm:pieri-case}}
\begin{thm:pieri-case}
Let $\beta$ be a horizontal strip and let $X = X_\eset^{\rect}(\alpha,\ybox,\beta,\gamma)$. There is a natural indexing of $X$ by the numbers $1, \ldots, |X|$, and under this indexing, the action of $\omega'$ is given by $\omega'(i) = i+1 \ (\mathrm{mod} \ |X|).$ In K-theory, $k_{\alpha \beta}^{\gamma^c} = |X| - 1$, and each increasing tableau corresponds to a successive pair $(X_i, X_{i+1})$ in the orbit, excluding the final pair $(X_{|X|},X_1)$.
\end{thm:pieri-case}
In this and certain other cases, the equations of Corollary \ref{cor:parity-eqn} in fact hold over $\ZZ$, and the inequality is an equality. In general, however, the quantity $c_{\alpha \tinybox \beta \gamma}^{\rect} - k_{\alpha \beta}^{\gamma^c}$ may be negative, and equality only holds mod 2. The author would be interested in combinatorial explanations of these facts.


\subsection{Acknowledgments} I am indebted first and foremost to David Speyer for introducing me to Schubert calculus (with and without osculating flags) and for many helpful conversations. Thanks also to Rohini Ramadas and Maria Gillespie, and to Oliver Pechenik for pointing out Example \ref{exa:disconnected}. Part of this work was done while supported by NSF grant DMS-1361789.

\subsection{Structure of this paper} The paper is as follows. In Section \ref{sec:schubert-mbar}, we give background on $\Mbar{r}$ and geometrical Schubert calculus; we then prove Theorem \ref{thm:main-geom}. In Section \ref{sec:tableau-combinatorics}, we give background on tableau combinatorics and dual equivalence. In Section \ref{sec:schubert-real}, we prove Theorem \ref{thm:DE-caterpillar-covering} and Corollary \ref{cor:promotion}. Finally, we discuss connections to K-theory in Section \ref{sec:k-theory}.

\section{Schubert problems over $\Mbar{r}(\CC)$} \label{sec:schubert-mbar}

\subsection{Grassmannians and Schubert varieties} \label{sec:grass-schubert}
We write $G(k,n)$ for the Grassmannian of dimension-$k$ vector subspaces of $\CC^n$, or $G(k,V)$ if we wish to specify an ambient vector space $V$.

Let $\lambda = (\lambda^{(1)} \geq \lambda^{(2)} \geq \cdots \geq \lambda^{(k)} \geq 0)$ be a partition with $k$ parts, each of size at most $n-k$. We write $\lambda \subseteq \rect$ to denote this. Let $\sF$ be a complete flag; let $F^j$ be the codimension-$j$ part of $\sF$. We define the \newword{Schubert variety}
\[\Omega(\lambda, \sF) = \{V \in G(k,n) : \dim(V \cap F^{k-j+\lambda^{(j)}}) \geq j\};\]
this is an integral subvariety of codimension $|\lambda| = \sum \lambda^{(j)}.$

The cohomology class of $\Omega(\lambda,\sF)$ does not depend on the choice of $\sF$; we write $[\Omega(\lambda)]$ for this class. It is well-known that the classes $\{ [\Omega(\lambda)]\}_{\lambda \subseteq \rect}$ form an additive basis for $H^*(G(k,n))$. Given partitions $\lambda_1, \ldots, \lambda_r$, we may write
\[[\Omega(\lambda_1)] \cdots [\Omega(\lambda_r)] = \sum_\nu c_{\lambda_\bullet}^\nu [\Omega(\nu)],\]
and we call the structure constants $c_{\lambda_\bullet}^\nu$ the Littlewood-Richardson numbers. (Note that $c_{\lambda_\bullet}^\nu = 0$ unless $|\nu| = \sum |\lambda_j|$.) We will occasionally use the identities
\[c_{\lambda_1, \ldots, \lambda_r}^{\rect} = c_{\lambda_1, \ldots, \lambda_{r-1}}^{\lambda_r^c},\]
where $\lambda_r^c$ denotes the complementary partition with respect to $\rect$,
\[\lambda^c = (n-k+1-\lambda^{(k)} \geq n-k+1- \lambda^{(k-1)} \geq \cdots \geq n-k+1- \lambda^{(1)}),\]
and the Pieri rule:
\[c_{\lambda, \tinybox}^{\mu} = \begin{cases} 1 &\text{if } |\mu| = |\lambda| + 1 \text{ and } \mu \supset \lambda, \\ 0 & \text{otherwise.} \end{cases} \]

\subsection{Linear systems and higher ramification} \label{subsec:lin-sys} We fix the following notation: for an integral projective curve $X$ of genus 0, let $G(k,n)_X = G(k,H^0(\cO_X(n-1)))$. The points $V \in G(k,n)_X$ parametrize projections from the degree $(n-1)$ Veronese embedding,
\[X \into \PP(H^0(\cO_X(n-1))^\vee) \dashrightarrow \PP(V^\vee) = \PP^{k-1},\]
that is, morphisms $X \to \PP^{k-1}$ of degree at most $n-1$.

For $p \in X$, we define the \newword{osculating flag} $\sF(p)$ in $H^0(\cO_X(n-1))$ by
\[\sF(p)^j = \{D \in H^0(\cO_X(n-1)) : D - j[p] \text{ is effective}\}.\]
Geometrically, $\sF(p)$ is dual to the unique flag $\sH$ whose projectivizations intersect $X$ at $p$ with the highest possible multiplicity. Explicitly, $\sH$ is given by the projective planes
\[\HH_j = \PP\left(\left(H^0(\cO_X(n-1))/\sF(p)^{j+1}\right)^\vee\right).\]
This is the unique plane of (projective) dimension $j$ that intersects $X$ at $p$ with multiplicity $j+1$ in the Veronese embedding. Thus $\HH_0 = p$, $\HH_1$ is the tangent line to $X$ at $p$, and so on. In coordinates, the embedding is 
\[[z:1] \mapsto [1 : z : z^2 : \cdots : z^{n-1}]\] and $\sH$ is given by the top rows of the $n \times n$ matrix
\begin{equation} \label{eqn:flag-matrix}
\begin{bmatrix}
\frac{d^i}{dz^i}(z^{j-1})
\end{bmatrix}
=
\begin{bmatrix}
1 & z & z^2 & \cdots & z^{n-1} \\
0 & 1 & 2z & \cdots & (n-1) z^{n-2} \\
0 & 0 & 2 & \cdots & (n-1)(n-2) z^{n-3} \\
\vdots & \vdots & \vdots &\ddots & \vdots \\
0 & 0 & 0 & \cdots & (n-1)!
\end{bmatrix}.
\end{equation}
Schubert conditions with respect to $\sF(p)$ or $\sH$ are called \emph{higher ramification conditions} at $p$ for the map $X \to \PP^{k-1}$. We will only consider Schubert varieties with respect to osculating flags, so by abuse of notation we will write $\Omega(\lambda,p)$ for the Schubert variety with respect to $\mathscr{F}(p)$ in $G(k,n)_X$. Given points $p_1, \ldots, p_r$ on $X$ and partitions $\lambda_1, \ldots, \lambda_r$, we define the Schubert problem
\[S(\lambda_\bullet, p_\bullet) = \bigcap_{i=1}^r \Omega(\lambda_i,p_i).\]
We will sometimes think of a point $x \in S(\lambda_\bullet,p_\bullet)$ as a morphism $X \to \PP^{k-1}$ with prescribed ramification conditions $\lambda_i$ at $p_i$ for each $i$. We have the \newword{Pl\"{u}cker Formula}, which says that the total amount of ramification of a linear series $V$ is always equal to $\dim_\CC G(k,n)$:

\begin{thm}[Pl\"{u}cker formula]
Let $V \in G(k,n)_X$. For each $x \in X$, let $\lambda_x$ be the largest Schubert condition such that $V \in \Omega(\lambda_x,x)$. Then $\lambda_x = \eset$ for all but finitely-many $x$, and
\[\sum_{x \in X} |\lambda_x| = k(n-k).\]
\end{thm}
See \cite{GrHaBook} for a proof. When $k=2$, the Pl\"{u}cker formula reduces to the Riemann-Hurwitz formula for ramification points of maps $\mathbb{P}^1 \to \mathbb{P}^1$ of degree $n-1$. The Pl\"{u}cker formula is essentially equivalent to Theorem \ref{thm:EH86} of \cite{EH86}, that the dimension of $S(\lambda_\bullet,p_\bullet)$ is always $k(n-k) - \sum |\lambda_i|$. Here it is helpful to note that $\Omega(\ybox, \sF)$ is an ample divisor.

Finally, we have the following formula for minors of the matrix \eqref{eqn:flag-matrix} above and the Schubert condition $\ybox$:
\begin{lemma} \label{lem:local-equation-box}
For a subset $J \subset [n]$, let
\begin{align*}
\Delta_J &= \prod_{j_1 < j_2 \in J} (j_2 - j_1), \\
e_J &= \sum_{j \in J} j - {|J| + 1 \choose 2}.
\end{align*}
Note that $e_J + e_{J^c} = |J|(n-|J|).$ We have
\[\Delta_J(z) = \Delta_J \cdot z^{e_J},\]
where $\Delta_J(z)$ is the determinant of the top-justified square minor of the matrix \eqref{eqn:flag-matrix} using the columns $J$.  A point $V \in G(k,n)_X$, with Pl\"{u}cker coordinates $pl_I$, satisfies the Schubert condition $\ybox$ with respect to the flag \eqref{eqn:flag-matrix} if and only if
\[\sum_{I \in {[n] \choose k}} (-1)^{e_{I_c}}\Delta_{I^c}(z) \cdot pl_I = \sum_{I \in {[n] \choose k}} \Delta_{I^c} \cdot (-z)^{e_{I_c}} \cdot pl_I = 0.\]
\end{lemma}
\begin{proof}
See \cite{Pu}.
\end{proof}

\subsection{Curves with marked points} \label{sec:M0r} A \newword{nodal curve} is a connected, reduced projective variety $C$ of dimension 1, all of whose singularities are simple nodes. We define the \newword{dual graph} of $C$ to be the graph $G = (V,E)$, where
\[V = \{\text{irreducible components of } C\},\  E = \{\text{nodes of } C\}.\]
We say $C$ has arithmetic genus $g = \dim_\CC H^1(C,\cO_C)$. We are interested in curves of genus zero, and we note that $C$ is genus zero if and only if every irreducible component of $C$ is isomorphic to $\PP^1$ (in particular, is smooth) and the dual graph of $C$ is a tree.

We select distinct smooth points $p_1, \ldots, p_r$ on $C$, and consider $C$ up to automorphisms $\phi$ that fix the $p_i$. We say $C$ is \newword{stable} if the only automorphism of $C$ fixing the marked points is the identity. Since $\mathrm{Aut}(\PP^1)$ is simply 3-transitive, $C$ is stable if and only if every component of $C$ has $\geq3$ nodes and/or marked points. We say $p \in C$ is a \newword{special point} if it is a node or a marked point.

We define
\begin{align*}
\cU_r &= \{(p_1, \ldots, p_r) : p_i \ne p_j \text{ for all } i \ne j\} \subset \PP^1 \times \cdots \times \PP^1, \\
\Mo{r} &= \cU_r / \mathrm{Aut}(\PP^1),
\end{align*}
where the action of $\mathrm{Aut}(\PP^1)$ is the diagonal. We think of $\Mo{r}$ as the moduli space of \emph{irreducible} stable curves with $r$ distinct marked points. We have an open immersion
\[\Mo{r} \hookrightarrow \Mbar{r} = \{\text{stable, genus-0 curves with $r$ distinct smooth marked points}\} / \sim,\]
where two curves $(C, p_\bullet)$ and $(C', p'_\bullet)$ are equivalent if there is an isomorphism $\phi : C \to C'$ such that $\phi(p_i) = p'_i$. The space $\Mbar{r}$ is a smooth projective variety of dimension $r-3$, with a universal family $\cC \to \Mbar{r}$, flat and of relative dimension 1, where the fiber over the point $[C]$ is the curve $C$ itself. 

We note that $\Mbar{r}(\CC)$ has a stratification into locally closed cells indexed by trees $T$ with $r$ labeled leaves, such that every internal vertex has degree $\geq 3$, up to graph isomorphism preserving the leaf labels. The cell corresponding to $T$ is the set of stable curves whose dual graph is $T$; it has dimension
\[\sum_{\text{internal vertices } v \in T} (\deg(v) - 3).\] The unique maximal cell, corresponding to the graph with only one internal vertex, is $\Mo{r}$. The 0-dimensional cells of $\Mbar{r}$ correspond to `minimally stable' curves $C$, where each component has exactly 3 special points. If $C$ is minimally stable and the internal nodes of its dual graph form a line (i.e. there are no components having 3 nodes), we say $C$ is a \newword{caterpillar curve}. Caterpillar curves will play a special combinatorial role in this paper.

\begin{figure}[h]
\centering
\includegraphics[scale=0.5]{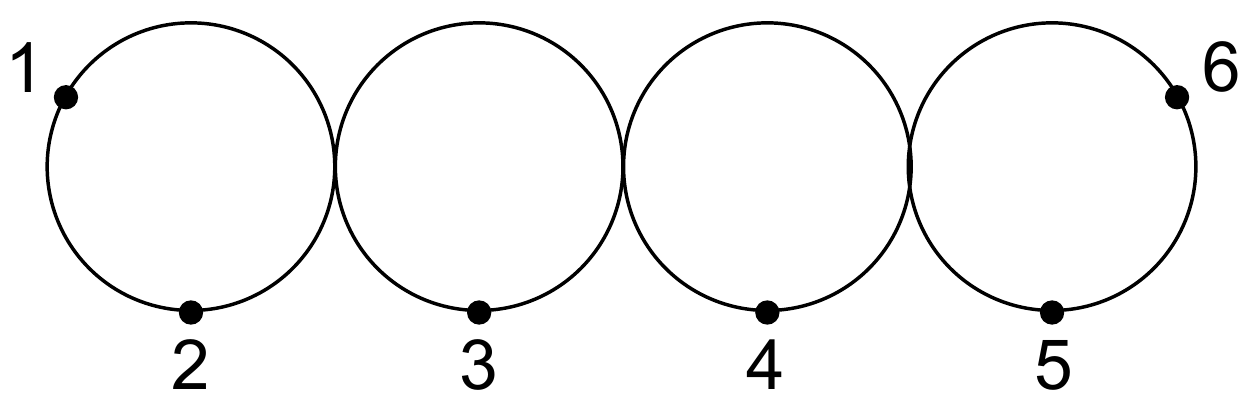}
\caption{A caterpillar curve with 6 marked points.}
\label{fig:caterpillar-curve}
\end{figure}

\subsubsection{Forgetting maps}

Let $T = \{1, \ldots, n\}$ and let $T' \subseteq T$. We define the \newword{forgetting map} $\varphi_{T'} : \Mbar{T} \to \Mbar{T - T'}$ as follows: given $[C] \in \Mbar{T}$, we forget the points with labels in $T'$; then we contract any irreducible component of $C$ that is left with fewer than three special points. This gives a stable curve with marked points labeled by $T - T'$. See Figure \ref{fig:forgetful}.

\begin{figure}[h]
\centering
\includegraphics[scale=0.5]{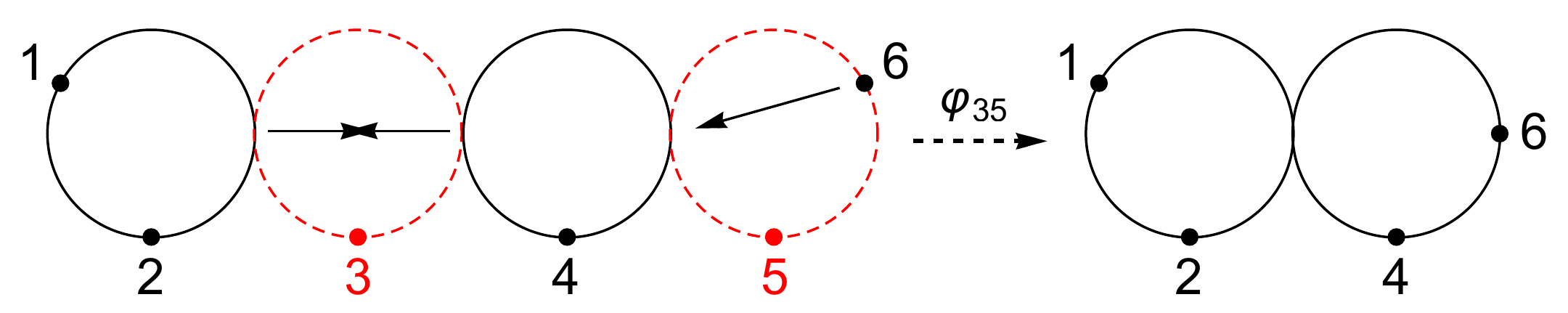} \\
\caption{If we forget the points labeled 3 and 5, we contract their components.}
\label{fig:forgetful}
\end{figure}

The simplest forgetting map $\varphi_{r+1} : \Mbar{r+1} \to \Mbar{r}$ is of special importance: the fiber over $[C] \in \Mbar{r}$ is a copy of $C$ itself. We may think of the $(r+1)$-st marked point as moving around $C$, bubbling off new components when it collides with the existing special points. Thus $\Mbar{r+1}$ is isomorphic to the universal family $\cC \to \Mbar{r}$.

\subsubsection{Topology of $\Mbar{r}$ over $\RR$}

The locus $\Mbar{r}(\RR)$ is a smooth manifold of (real) dimension $r-3$ with the structure of a CW-complex, which we now describe. We refer to \cite{Devadoss} for this material. First, a \newword{dihedral ordering} of a finite set $X$ is an equivalence class of circular orderings $<$ on $X$, where $<_1$ and $<_2$ are equivalent if they are opposites, that is
\[a <_1 b \text{ if and only if } b <_2 a.\]
In other words, a dihedral ordering is a circular ordering, up to reflection. The cells of $\Mbar{r}(\RR)$ are indexed by the following data:
\begin{enumerate}
\item An unrooted tree $T$ with $r$ labeled leaves, up to isomorphism, as above;
\item For each internal vertex $v \in T$, a dihedral ordering of the edges incident to $v$.
\end{enumerate}
The dihedral ordering arises from the fact that $\mathrm{Aut}(\RR\PP^1)$ acts only by rotating and reflecting the marked points. There are $\tfrac{1}{2}(r-1)!$ maximal cells of real dimension $r-3$, corresponding to the dihedral orderings of $r$ points on $\RR\PP^1$. The codimension-one cells correspond to curves with exactly one node; we will speak of \newword{wall-crossing} from one maximal cell to an adjacent one, which results in reversing the order of a consecutive sequence of points.

\subsection{Node labelings.}
Now let $C$ be a stable curve with components $C_i$ and marked points $p_1, \ldots, p_r$. A \newword{(strict) node labeling} $\nu$ of $C$ is a function
\[\nu : \big\{(q,C_i) : q \text{ is a node on } C_i\big\} \to \big\{ \text{partitions } \lambda \subseteq \rect \big\},\]
such that if $q$ is the node between $C_i$ and $C_j$, then
\[\nu(q,C_i) = \nu(q,C_j)^c,\]
where $\nu^c$ denotes the complementary partition. This is a choice of a pair of complementary partitions on opposite sides of each node. We will also consider \newword{excess node labelings}, where instead we allow $\nu(q,C_i) \supseteq \nu(q,C_j)^c$, i.e. the partitions may be more than complementary. Given a node labeling $\nu$, we define the space
\[\Phi_\nu = \prod_{\text{components } C_i}\ \bigcap_{\text{nodes } q \in C_i} \Omega(\nu(q,C_i),q) \subseteq \prod_{i} G(k,n)_{C_i},\]
so $\Phi_\nu$ applies the Schubert conditions from  $\nu$ separately in each $G(k,n)_{C_i}$. All our Schubert problems on $C$ take place in the ambient space
\begin{equation} \label{eqn:node-labeling-space}
\cG(k,n)_C = \bigcup_{\text{strict node labelings } \nu} \Phi_\nu \ \ \subset\ \prod_i G(k,n)_{C_i}.
\end{equation}
We note that this is the space of limit linear series on $C$, in the sense of Eisenbud-Harris \cite{EH86}.

Let $\lambda_i\ (1 \leq i \leq r)$ be a choice of partition for each marked point of $C$. Let $C_{p_i}$ be the component containing $p_i$ and $\Omega(\lambda_i,p_i) \subset G(k,n)_{C_{p_i}}$ be the Schubert variety in the appropriate Grassmannian. We define the Schubert problem on $C$,
\[\cS(\lambda_\bullet)_C = \cG(k,n)_C \cap \bigg( \bigcap_{i=1}^r \Omega(\lambda_i,p_i) \bigg),\]
Thus the components of $\cS(\lambda_\bullet)_C$ are precisely the components of
\[\Phi_\nu \cap \bigg( \bigcap_{i=1}^r \Omega(\lambda_i,p_i) \bigg),\]
for all (strict) node labelings $\nu$ of $C$. Our Schubert problems therefore describe collections of morphisms $C_i \to \PP^{k-1}$ with prescribed ramification at the nodes and marked points of $C_i$.

Speyer has shown that the above space moves in families of marked stable curves:
\begin{thm}[Theorem 1.1 of \cite{Sp}] \label{thm:recall-Sp14}
There exist flat, Cohen-Macaulay families $\cG(k,n)$ and $\cS(\lambda_\bullet)$ over $\Mbar{r}$, with an inclusion
\[
\xymatrix{
\cS(\lambda_\bullet) \ar[dr] \ar@{^{(}->}[r] & \cG(k,n) \ar[d] \\
& \Mbar{r}.
}\]
The relative dimensions of $\cG(k,n)$ and $\cS(\lambda_\bullet)$ are $k(n-k)$ and $k(n-k) - \sum|\lambda_i|$, and for each point $[C] \in \Mbar{r}$, the fibers are the spaces $\cG(k,n)_C$ and $\cS(\lambda_\bullet)_C$ described above.
\end{thm}

We sketch the construction. First, for each subset $T \subseteq [n]$ of size 3, we have a forgetting map $\varphi_T : \Mbar{r} \to \Mbar{T}$, and a Grassmannian $G(k,n)_T$. We pull these back to $\Mbar{r}$ and form a large fiber product
\[\sB = \prod_{\varphi_T} G(k,n)_T.\]
This is a trivial bundle over $\Mbar{r}$, with fibers isomorphic to products of ${r \choose 3}$ copies of $G(k,n)$. Over $\Mo{r}$, we have the trivial bundle
\[G(k,n)_{\PP^1} \times \Mo{r} \to \Mo{r},\]
with a diagonal embedding
\[\Delta: G(k,n)_{\PP^1} \hookrightarrow \sB\]
commuting with the projection to $\Mo{r}$. Then $\cG(k,n)$ is the closure of the image of $\Delta$. A detailed analysis of the factors in $\sB$ then establishes that $\cG(k,n)$ is flat and Cohen-Macaulay and the boundary fibers of $\cG(k,n)$ have the desired form. 

An important element of the construction is the following. Let $[C] \in \Mbar{r}$ be a stable curve. For each irreducible component $C_i \subset C$, there is a factor $G(k,n)_{T_i}$, such that, over a neighborhood $U$ of $[C]$, the projection
\[\pi : \sB \to \prod_i G(k,n)_{T_i}\]
gives an isomorphism of $\cG(k,n)$ onto its image. Moreover, this isomorphism identifies $G(k,n)_{T_i}$ with the Grassmannian $G(k,n)_{C_i}$ defined above. In particular, this embeds the fiber of $\cG(k,n)$ over $[C]$ into $\prod_{C_i} G(k,n)_{C_i}$, where it is the space $\cG(k,n)_C$ of Equation \eqref{eqn:node-labeling-space}. (The same is true of $\cS(\lambda_\bullet)$.) We will use this fact in our proof of Theorem \ref{thm:main-geom}.

\subsubsection{Excess node labelings} The irreducible components of $\cS(\lambda_\bullet)$ are described by strict node labelings, but we must also consider excess node labelings. They will arise in two ways:
\begin{itemize}
\item by intersecting components of $\cS(\lambda_\bullet)$ described by different node labelings, and
\item by the forgetting maps $\Mbar{r} \to \Mbar{s}$ with $s < r$.
\end{itemize}
For the first, consider two node labelings $\nu$ and $\nu'$ and the corresponding subsets of $S_\nu, S_{\nu'} \subseteq \cS(\lambda_\bullet,p_\bullet)$. Then we have
\[S_\nu \cap S_{\nu'} = S_{\nu \cup \nu'},\]
where $\nu \cup \nu'$ is the excess node labeling obtained by taking the union of the labels of $\nu$ and $\nu'$. Note that this intersection is nonempty if and only if, for each component $C_i \subset C$, the excess Schubert problem from $\nu \cup \nu'$ on $C_i$ is nonempty.

For the second, let $\nu$ be a node labeling on $C$. Consider a forgetting map $\Mbar{r} \to \Mbar{s}$. Let $C'$ be the image of $C$.

\begin{lemma}Assume that $S_\nu$ is nonempty. Then the labeling of the nodes of $C'$ by the same labels as $C$ (on the remaining components) is an excess node labeling.
\end{lemma}

\begin{proof}
By forgetting points one at a time, it is sufficient to consider the case $s = r-1$. In this case, at most one component is contracted. Call it $Z$, and assume $Z$ is connected to two other components $X,Y$. (If $Z$ is connected to only one other component, the node vanishes when we contract $Z$, so there is nothing to prove.)
Let $q_X, q_Y$ be the pair of nodes connecting $Z$ to $X,Y$. So we have
\[\nu(q_X,X) \supseteq \nu(q_X,Z)^c \text{ and } \nu(q_Y,Z) \supseteq \nu(q_Y,Y)^c.\]

By definition, in $C'$, the node between $X$ and $Y$ has labels $\nu(q_X,X)$ and $\nu(q_Y,Y)$. Since we assumed $S_\nu$ was nonempty (for $C$), the Schubert problem on $Z$ must be nonempty, so
\[\nu(q_X,Z)^c \supseteq \nu(q_Y,Z),\]
which gives the desired containment $\nu(q_X,X) \supseteq \nu(q_Y,Y)^c.$
\end{proof}

\begin{rmk}
In the case where $Z$ is connected to only one other component $X$, the second special point on $Z$ must be a marked point $p$. If $p$ is labeled by $\lambda$, the same proof shows that $\lambda \subseteq \nu(q_X,X)$, so our contraction procedure also produces an excess Schubert condition at $p$.
\end{rmk}

Finally, we will use the fact that any excess node labeling on $C$ comes from contracting a strict node labeling with additional marked points:

\begin{lemma}
Let $\nu$ be an excess node labeling. Assume there is only one node $q = X \cap Y$ with excess labels. Then there is a unique curve $\tilde C$ with $r+1$ marked points, and a unique \emph{strict} node labeling $\tilde \nu$ of $\tilde C$, such that forgetting the $(r+1)$-st point takes $\tilde C$ to $C$ and $\tilde \nu$ to $\nu$.
\end{lemma}

\begin{proof}
Let $\tilde C$ be the curve in which $q$ is replaced by an extra component $Q$, having nodes $q_X$ and $q_Y$, and mark $Q$ with an $(r+1)$-st marked point $p_{r+1}$. Define $\tilde \nu$ to be the same as $\nu$ for nodes other than $q_X$ and $q_Y$, and set
\begin{align*}
\tilde \nu(q_X,X) = \nu(q,X) &\text{ and } \tilde\nu(q_X,Q) = \nu(q,X)^c, \\
\tilde \nu(q_Y,Y) = \nu(q,Y) &\text{ and } \tilde\nu(q_Y,Q) = \nu(q,Y)^c.
\end{align*}
(See Figure \ref{fig:excess-pop}.) It is clear that $\tilde C$ contracts to $C$ and $\tilde \nu$ to $\nu$ under the forgetting map $\varphi_{r+1}$, and that this construction is unique.
\end{proof}

\begin{figure}[h]
\centering
\includegraphics[scale=0.5]{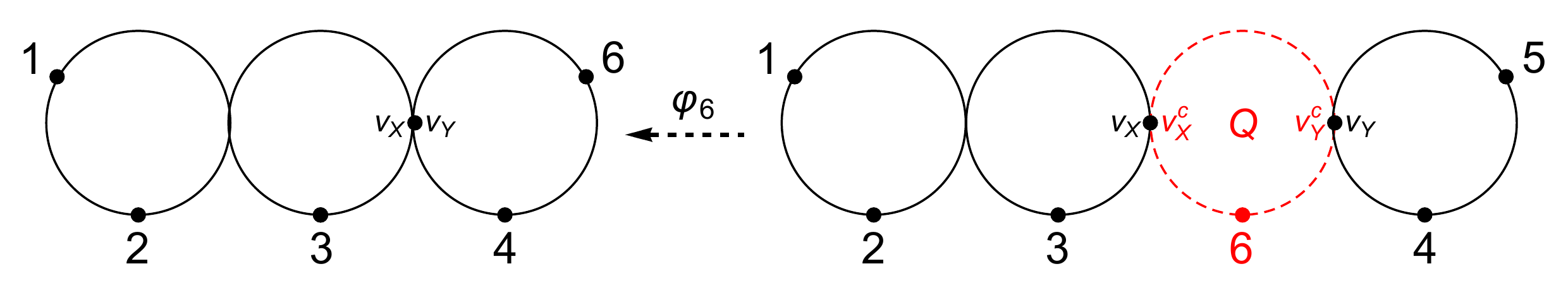}
\caption{Turning an excess node labeling into a strict node labeling (of a larger curve).}
\label{fig:excess-pop}
\end{figure}

\subsubsection{The dimension-1 case} We now assume $\sum |\lambda_i| = k(n-k)-1$, so $\cS(\lambda_\bullet,p_\bullet)$ has dimension 1.

For each node labeling $\nu$, precisely one component $C_\nu$ of $C$ has labels that sum to $k(n-k)-1$; all other components have labels that sum to $k(n-k)$. We will call $C_\nu$ the \emph{main component} of $C$ and the other $C_i$'s the \emph{frozen components} for the node labeling $\nu$. We have the following description of the connectivity between $S_\nu$ for different $\nu$'s:

\begin{lemma} \label{excess-labels-dim1}
Let $\nu$ and $\nu'$ be distinct strict node labelings, and suppose $S_\nu \cap S_{\nu'}$ is nonempty. Then the main components $C_\nu$ and $C_{\nu'}$ of $C$ are distinct and adjacent, $\nu$ and $\nu'$ agree everywhere except at the node $q = C_\nu \cap C_{\nu'}$, and $\nu'(q,C_\nu)$ is an extension of $\nu(q,C_\nu)$ by exactly 1 box (and vice versa for the labels on $C_{\nu'}$.)
\end{lemma}

\begin{proof}
If $C_i$ is a ``frozen'' component, the labels on it cannot change: otherwise the Schubert problem on $C_i$ will be overdetermined and $S_\nu \cap S_{\nu'}$ will be empty. Hence $\nu$ and $\nu'$ agree on any component  which is frozen for both. Moreover, if $q$ is a node and one side of $q$ is frozen for both labelings, then $\nu$ and $\nu'$ agree on the frozen side, hence on both sides (since the labelings are strict). In particular, if the main components $C_\nu, C_{\nu'}$ are equal or non-adjacent, every node must have at least one side on a shared frozen component, hence $\nu = \nu'$, a contradiction.

The only case remaining is where there exists a node $q$ between the main components. Since $C_\nu$ is frozen for $\nu'$, we see that $\nu(q,C_\nu) \cup \nu'(q,C_\nu) = \nu'(q,C_\nu)$; by counting, the latter is exactly one box larger than $\nu(q,C_\nu)$.
\end{proof}

\subsection{Lifting to $\Mbar{r+1}$.} Consider the following observation: let $p_1, \ldots, p_r$ be distinct points on $\PP^1$, and let $x \in S(\lambda_\bullet,p_\bullet) \subset G(k,n)$ be a point of the solution to the corresponding Schubert problem. So $x$ induces a morphism $f : \PP^1 \to \PP^{k-1}$ with higher ramification as specified by $\lambda_\bullet$.

We have prescribed only $k(n-k)-1$ worth of higher ramification of $f$. Hence, by the Pl\"{u}cker formulas, there exists a unique point having additional ramification by 1 box. (Either some $p_i$ satisfies a one-box-larger Schubert condition $\lambda_i'$, or some unmarked point $z \in \PP^1$ is simply ramified and $x \in S(\lambda_\bullet, p_\bullet) \cap S(\ybox, z)$ for a unique $z$.) Let $S'$ be the incidence correspondence
\begin{equation}
\label{incidence-correspondence} S' = \{(x,z) : x \in \cS(\lambda_\bullet,p_\bullet)) \cap S(\ybox, z)\} \subset G(k,n) \times (\PP^1 - \{p_1, \ldots, p_r\}).
\end{equation}
The projection to $G(k,n)$ induces a map $\pi: S' \to S(\lambda_\bullet, p_\bullet)$; by the above remarks, $\pi$ is injective. In fact, letting $\overline{S'} \subset G(k,n) \times \PP^1$ be the closure, we will show that $\pi : \overline{S'} \to S(\lambda_\bullet,p_\bullet)$ is an isomorphism, and remains so when the $p_i$ (and $z$) are allowed to collide. (Note that we are not assuming $S(\lambda_\bullet,p_\bullet)$ to be smooth.)

We will need the following lemma on simple nodes:

\begin{lemma} \label{nodal-iso} Let $X,Y \subseteq Z$ be subschemes such that $X \cup Y = Z$ and the scheme-theoretic intersection $X \cap Y$ is one reduced point. Let $f : A \to Z$ be a morphism whose restrictions $f^{-1}(X) \to X$ and $f^{-1}(Y) \to Y$ are isomorphisms. Then $f$ is an isomorphism.
\end{lemma}
\begin{proof}
See Corollary \ref{nodal-iso-appendix} in the appendix.
\end{proof}

\begin{thm} \label{box-lift}
With notation as above, let $z$ be an $(r+1)$-st marked point; label $z$ with a single box. Composing the ``forgetting" map $\varphi_{r+1} : \Mbar{r+1} \to \Mbar{r}$ with the structure map for $\cS(\lambda_\bullet;\ybox_z)$ yields the following diagram:
\[
\xymatrix{
\cS(\lambda_\bullet;\ybox_z) \ar[r] \ar[d]_{\pi} & \Mbar{r+1} \ar[d]^{\varphi_{r+1}} \\
\cS(\lambda_\bullet) \ar[r] & \Mbar{r}
}
\]
(This diagram is not Cartesian.) Then $\pi$ is an isomorphism.
\end{thm}
If $\tilde x \in \cS(\lambda_\bullet;\ \ybox_z)$ lying over $\tilde C \in \Mbar{r+1}$, the map $\pi$ consists of forgetting the marked point $z$, then possibly contracting the component $Z \subset \tilde C$ containing $z$. In the latter case, $\pi(\tilde x)$ also forgets the morphism $Z \to \PP^{k-1}$, which had ramification exactly $\ybox$ at $z$. Thus we must recover both $z$ and, when necessary, the additional morphism.

\begin{proof}
We first construct the set-theoretic inverse for $\pi$. Let $x \in \cS(\lambda_\bullet)$, lying over a stable curve $C$, and let $\nu$ be a node labeling of $C$ with $x \in S_\nu$. Let $C_\nu \subseteq C$ be the main component of $\nu$, so $x$ gives a morphism $C_\nu \to \PP^{k-1}$ for which all but one point of ramification has been specified. Let $t \in C_\nu$ be the point with additional ramification. (Note that $t$ does not depend on the choice of $\nu$: if $x \in S_\nu \cap S_{\nu'}$, then the main components of $\nu$ and $\nu'$ are adjacent by Lemma \ref{excess-labels-dim1}, and $t$ must be the node between them, where the excess labels occur.) The assignment $x \mapsto t$ gives a (set-theoretic) diagonal map $\cS(\lambda_\bullet) \to \Mbar{r+1}$ that commutes with the diagram (thinking of $\Mbar{r+1}$ as the universal family over $\Mbar{r}$.) Let $\tilde C$ be the curve corresponding to $t \in \Mbar{r+1}$.

If $t$ is not a special point, then $\tilde C$ is the same curve as $C$, with $z=t$ as the $(r+1)$-st marked point. The morphisms $C_i \to \PP^{k-1}$ corresponding to $x \in \cS(\lambda_\bullet)$ already satisfy $\ybox$ at $t$, so they recover the point $\tilde x \in \cS(\lambda_\bullet;\ybox_z)$ lying over $\tilde C$. But if $t$ is a special point, $\tilde C$ has an additional component $Z$ bubbled off at $t$; we must recover the morphism $Z \to \PP^{k-1}$. There are two cases.

\emph{Case 1.} Suppose $t = p_i$. Then $Z$ has one node and two marked points $p_i$ and $z$. Let $\lambda_i^+$ be the (stricter) Schubert condition satisfied at $t$ for the map $C_\nu \to \PP^{k-1}$, so $\lambda_i^+/\lambda_i$ is one box. The morphism $Z \to \PP^{k-1}$ must satisfy $\lambda_i$ at $p_i$, $\ybox$ at $z$ and the strict node labeling condition $(\lambda_i^+)^c$ at the node with $C_\nu$. The Littlewood-Richardson coefficient $c_{\lambda_i\ybox(\lambda_i^+)^c}^{\rect}$ is 1, so there is a unique such morphism.

\emph{Case 2.} Suppose $t$ is a node between components $A,B$. Then $x$ satisfied an excess node labeling $\nu \cup \nu'$; let its excess labels at $t$ be $\alpha$ on $A$ and $\beta$ on $B$, so by Lemma \ref{excess-labels-dim1}, $\alpha \supset \beta^c$ and $\alpha/\beta^c$ is one box. Now $Z$ has two nodes and the marked point $z$, and the morphism $Z \to \PP^{k-1}$ must satisfy $\ybox$ at $z$, along with the strict node labeling conditions $\alpha^c$ at the node with $A$ and $\beta^c$ at the node with $B$. The Littlewood-Richardson coefficient $c_{\beta^c \ybox \alpha^c}^{\rect} = 1$, so again the morphism exists and is unique.

We now show that $\pi$ is an isomorphism. In particular, we show that for every point $[C] \in \Mbar{r}$, the restriction of $\pi$ in the diagram
\[
\xymatrix{
\cS(\lambda_\bullet; \ybox_z)\big|_C \ar[r] \ar[d]_{\pi} & C \ar[d]^{\varphi_{r+1}} \\
\cS(\lambda_\bullet)\big|_{[C]} \ar[r] & [C]
}
\]
is an isomorphism. (Recall that the fiber of the forgetting map over $[C]$ is $C$ itself.) In particular, it follows that for every $x \in \cS(\lambda_\bullet)$, the scheme-theoretic fiber $\pi^{-1}(x)$ is one reduced point; hence $\pi$ will be a (global) isomorphism.

\emph{Reduction to the case where $C$ has one component}. Let $\nu$ be a node labeling and let $S_\nu$ be the corresponding subscheme of $\cS(\lambda_\bullet)\big|_{[C]}$. For any frozen component $C_i$, the Schubert problem in $G(k,n)_{C_i}$ has a finite set of solutions. So, in the containment $S_\nu \subset \prod_i G(k,n)_{C_i}$, the coordinates in the $G(k,n)_{C_i}$ factors corresponding to frozen components are locally constant. In particular, projection to $G(k,n)_{C_\nu}$, where $C_\nu$ is the main component, is locally an isomorphism.

Let $\nu'$ be any other node labeling. We claim that the scheme-theoretic intersection $S_\nu \cap S_{\nu'}$ is reduced. By Lemma \ref{excess-labels-dim1}, if the intersection is nonempty, the main components $C_\nu$ and $C_{\nu'}$ are distinct and adjacent. Let $x \in S_\nu \cap S_{\nu'}$. We project to $G(k,n)_{C_\nu} \times G(k,n)_{C_{\nu'}}$; this is an isomorphism on a neighborhood of $x$. But locally, the projections $S_\nu$ and $S_{\nu'}$ are contained in transverse fibers: $S_\nu \subseteq G(k,n)_{C_\nu} \times \{\text{pt}\}$ and $S_{\nu'} \subset \{\text{pt}\} \times G(k,n)_{C_{\nu'}}$. Thus the scheme-theoretic intersection $S_\nu \cap S_{\nu'}$ is reduced at $x$.

It follows from Lemma \ref{nodal-iso} that $\pi$ is an isomorphism if and only if it is an isomorphism when restricted to each $S_\nu$. So, by forgetting all the marked points on the frozen components for $\nu$, and contracting down to the main component, we may assume that $C$ has only one component. So $C = \PP^1$ with distinct marked points $p_1,\ldots, p_r$, and $\cS(\lambda_\bullet)$ lives in the single Grassmannian $G(k,n)_C$.

\emph{Factoring $\pi$}. For each marked point $p \in C \cong \PP^1$, let $C_p$ be the component obtained when $z$ collides with $p$ and bubbles off. We have containments
\[\cS(\lambda_\bullet; \ybox_z)\big|_C \subset \cG(k,n) \subset G(k,n)_C \times \prod_p G(k,n)_{C_p} \times \PP^1_z,\]
and the map $\pi$ is the projection that forgets the $G(k,n)_{C_p}$ factors and the $z$ coordinate. Note that the projection from $G(k,n)_{C_p}$ gives an isomorphism everywhere except possibly at $z=p$.

We factor $\pi$ into two projections,
\[\xymatrix{
\cS(\lambda_\bullet; \ybox_z)\big|_C \ar[r]^-\alpha & S' \ar[r]^-\beta & \cS(\lambda_\bullet),
}\]
where $S' \subset G(k,n)_C \times \mathbb{P}^1$ is obtained by forgetting the $G(k,n)_{C_p}$ factors, but not the $z$ coordinate. The map $\beta : S' \to \cS(\lambda_\bullet)$ is the closure of the incidence correspondence \eqref{incidence-correspondence}. \\

\emph{The map $\beta$ is an isomorphism}. Choose coordinates on $\PP^1_z$ so that $p_1 = 0$ and $\infty$ is not a marked point; we restrict to the set $\AA^1 = \PP^1 - \{\infty\}$. With notation from Lemma \ref{lem:local-equation-box}, the equation for $S'$ is then
\[f(z) = \sum_{I \in {[n] \choose k}} pl_I \Delta_{I^c} \cdot (-z)^{e_{I^c}}.\]
The leading term of $f(z)$ is $pl_{[k]} \Delta_{[n] \setminus [k]}\cdot z^{k(n-k)}$; note that $pl_{[k]}$ is a unit since (over $\AA^1$) the Schubert condition $\ybox$ is never satisfied at $\infty$.

Now, since $S'$ satisfies the Schubert condition $\lambda_1$ at $z=0$, all the Pl\"{u}cker coordinates $pl_I$ are zero for $I > I(\lambda_1)$, where
\[I(\lambda) = (n-k+1, \ldots, n-1, n) - \lambda.\] In particular, the lowest-degree term (corresponding to $I = I(\lambda_1)$) is $z^{|\lambda_1|}$, so we see that $z^{|\lambda_1|}$ divides $f(z)$. Our choice of coordinates was arbitrary, so by the same logic applied to the other marked points, we see that $(z-p_i)^{|\lambda_i|}$ divides $f(z)$ for each $i = 1, \ldots, r$. This gives
\[f(z) = z^{|\lambda_1|}(z-p_2)^{|\lambda_2|} \cdots (z-p_r)^{|\lambda_r|}g(z),\]
and by inspection we see $g(z)$ is linear, with leading term $pl_{[k]} \Delta_{[n] \setminus [k]} \ z$. Now, on the open set $\PP^1 - \{p_1, \ldots, p_r\}$, we may invert the $(z-p_i)$ factors. Hence the equation for $S'$, the closure over this open set, is just $g(z)$. Since $pl_{[k]}$ is a unit, the equation $g(z) = 0$ gives an isomorphism. \\

\emph{The map $\alpha$ is an isomorphism}. We consider the maps
\[\xymatrix{
\cS(\lambda_\bullet; \ybox_z)\big|_C \ar[r]^-\alpha & S' \ar[r]^-{\pi_2} & \PP^1_z.
}\]
We know that $\alpha$ is an isomorphism except possibly over the points $z = p_i$ for each $i$. We restrict to $z = p_i$, so $z$ and $p_i$ bubble off on the component $C_{p_i}$. We may project away from all the $G(k,n)_{C_p}$ factors except the one corresponding to $G(k,n)_{C_{p_i}}$, so the map $\alpha$ is the projection of $G(k,n)_C \times G(k,n)_{C_{p_i}}$ onto its first component. The fiber of $S(\lambda_\bullet; \ybox_z)$ at $z=0$ is now a union of the form
\[\bigcup_{\nu} A_\nu \times B_\nu,\]
where $\nu$ is a node labeling and $A_\nu, B_\nu$ are the corresponding subschemes of $G(k,n)_C$ and $G(k,n)_{C_{p_i}}$. Let $q$ be the node; then $\nu(q,C)$ is an extension of $\lambda_i$ by one box; in particular, the Littlewood-Richardson coefficient $c_{\lambda_i, \ybox}^{\nu(q,C_{p_i})}$ is 1, so
\[B_\nu = \Omega(\lambda_i,p_i) \cap \Omega(\ybox,z) \cap \Omega(\nu(q,C_{p_i}),q) \subset G(k,n)_{C_{p_i}}\]
is one reduced point. Thus the fiber is in fact of the form
\[\bigcup_\nu A_\nu \times \{pt\},\]
so the projection to the first factor is an isomorphism.
\end{proof}



Theorem \ref{thm:main-geom} now follows from Speyer's description of $\cS(\lambda_\bullet,\ybox_z) \to \Mbar{r+1}$. We obtain, as a corollary, our theorem on reality of curves over $\Mo{r}(\RR)$:
\begin{cor} \label{cor:as-real-as-poss}
If the $p_i$ are all in $\mathbb{RP}^1$, the curve $S = S(\lambda_\bullet, p_\bullet) \subset G(k,n)$ has smooth real points. Moreover, $S(\mathbb{C}) - S(\mathbb{R})$ is disconnected.
\end{cor}
\begin{proof}
We have a map $f: S \to \PP^1$. By Theorem \ref{thm:Sp14}, $S(\RR) \to \RR\PP^1$ is a covering map; in particular $S(\RR)$ is smooth. Also, since the preimage of every point $z \in \RR\PP^1$ consists of real points, we have $f^{-1}(\RR\PP^1) = S(\RR)$. Let $H^+, H_{-} \subset \CC$ be the (strict) upper and lower half-planes. Then $S$ is disconnected by its real points since 
\[S(\CC) - S(\RR) = f^{-1}(H^+) \sqcup f^{-1}(H_-). \qedhere\]
\end{proof}

For our applications to K-theory, we also need the following slightly stronger statement, in the case where $S$ is singular or reducible:
\begin{cor} \label{cor:normalization-disconnected}
Let $S' \subset S$ be any irreducible component, and let $\pi : \widetilde{S'} \to S'$ be its normalization. Then $\widetilde{S'}(\RR)$ is nonempty and $\widetilde{S'}(\CC) - \widetilde{S'}(\RR)$ is disconnected.
\end{cor}
\begin{proof}
Since $f : S \to \PP^1$ is flat, the map $S' \to \PP^1$ is surjective, and the fibers over $\RR\PP^1$ are all smooth real points of $S'$. Thus $\widetilde{S'}(\RR) = S'(\RR) \ne \eset$. The argument above, applied to $f \circ \pi$, shows that $\widetilde{S'}(\CC) - \widetilde{S'}(\RR)$ is disconnected.
\end{proof}

\section{Tableau combinatorics}\label{sec:tableau-combinatorics}

\subsection{Young tableaux and growth diagrams} \label{sec:combo-gds}
We recall the notion of a growth diagram of partitions. Let $G$ be the directed grid graph with vertices $\ZZ \times \ZZ$ and edges pointing up and to the right. We use the Cartesian convention for coordinates, so $(i,j)$ is $i$ steps to the right and $j$ steps up from the origin.

An induced subgraph $D \subset G$ is \emph{convex} if whenever $(a,b),(c,d) \in D$, the rectangle $(a,b) \times (c,d) \subseteq D$. A \newword{growth diagram on $D$} is a labeling $\lambda_{ij}$ of the vertices of $D$ by partitions, such that
\begin{enumerate}
\item[(i)] For each directed edge $\alpha \to \beta$, $\beta$ is an extension of $\alpha$ by a single box;
\item[(ii)] For each square
\[\xymatrix{
\alpha \ar[r] & \beta \\
\gamma \ar[u] \ar[r] & \delta, \ar[u]
}\]
if the two boxes of $\beta/\gamma$ are nonadjacent, then $\alpha$ and $\delta$ are the two distinct intermediate partitions between $\gamma$ and $\beta$.
\end{enumerate}
We think of (i) as the `growth condition' and (ii) as a `recurrence condition', for the following reason:
\begin{lemma}
Let $D$ be the rectangle $[a,b] \times [c,d]$ and let $\lambda_{ij}$ be a choice of partitions along a single path connecting $(a,b)$ to $(c,d)$. Then $\lambda_{ij}$ extends to a unique growth diagram on $D$.
\end{lemma}
\begin{proof}
Repeated application of condition (ii) uniquely specifies the remaining entries.
\end{proof}

Growth diagrams encode the \emph{jeu de taquin} (JDT) algorithm, as follows. Let $S,T$ be skew standard tableaux such that $T$ extends $S$. Let $\sh(S) = \beta/\alpha$ and $\sh(T) = \gamma/\beta$. We think of $S$ as a sequence $\alpha \subset \alpha_1 \subset \cdots \subset \alpha_n = \beta$ of partitions, where for each $i$, $\alpha_i/\alpha_{i-1}$ is the box of $S$ labeled $i$. Likewise, we will think of $T$ as a sequence of partitions $\beta_j$  growing from $\beta$ to $\gamma$.

Let $D$ be a rectangular grid of size $|\gamma/\beta| \times |\beta/\alpha|$. Label the left side of $D$ with the partitions $\alpha_i$ for $S$ and the top with the $\beta_j$ for $T$. Let $\widetilde{T}$ (resp. $\widetilde{S}$) be the bottom (resp. right) edges of the resulting growth diagram, thought of as skew standard tableaux.

\begin{lemma}
The tableau $\widetilde{S}$ is the result of applying forward JDT slides to $S$ in the order indicated by the entries of $T$ (starting with the smallest entry). The tableau $\widetilde{T}$ is the result of applying reverse slides to $T$ in the order indicated by the entries of $S$ (starting with the largest entry).

\end{lemma}
\begin{proof}
See \cite{Hai}.
\end{proof}
In this case we say $(S,T)$ \newword{shuffled} to $(\widetilde{T},\widetilde{S})$. We say $T$ is \newword{slide equivalent} to $\widetilde{T}$, and likewise $S$ is slide equivalent to $\widetilde{S}$.
\begin{lemma}
Shuffling is an involution.
\end{lemma}
\begin{proof}
The transpose of the growth diagram $D$ used to shuffle $(S,T)$ is again a growth diagram, with left and top edges $(\widetilde{T},\widetilde{S})$ and bottom and right edges $(S,T)$.
\end{proof}

We will be interested in growth diagrams on the downwards-slanting diagonal region
\[D = \{(i,j) : 0 \leq i + j \leq r\},\]
where every vertex on the main diagonal is labeled $\eset$, and every vertex on the outer diagonal is labeled by the rectangle $\rect$. We call these \newword{cylindrical growth diagrams}, for the following reason:

\begin{lemma}\label{cylindrical-recurrence}
Let $\lambda_{ij}$ be a cylindrical growth diagram. Then
\begin{itemize}
\item[(i)] $\lambda_{(i+r)(j-r)} = \lambda_{ij}$, and
\item[(ii)] $\lambda_{(r-j)(-i)} = \lambda_{ij}^c$.
\end{itemize}
Here $\lambda^c$ denotes the complementary partition with respect to $\rect$.
\end{lemma}

\begin{proof}
See \cite[Chapter 7, Appendix 1]{StanEC2}. Note that this fact is often attributed to Sch\"{u}tzenberger.
\end{proof}
Thus the rows of a cylindrical growth diagram repeat with period $r$; we may think of them as wrapping around a cylinder.

\subsection{Dual equivalence} Let $S,S'$ be skew standard tableaux of the same shape. We say $S$ is \newword{dual equivalent} to $S'$ if the following is always true: let $T$ be a skew standard tableau whose shape extends, or is extended by, $\sh(S)$. Let $\widetilde{T}, \widetilde{T}'$ be the results of shuffling $T$ with $S$ and with $S'$. Then $\widetilde{T} = \widetilde{T}'$.

In other words, $S$ and $S'$ are dual equivalent if they have the same shape, and they transform \emph{other} tableaux the same way under JDT.

\begin{lemma} \label{lem-dual-def2}
Let $S,S'$ be skew standard tableaux of the same shape. Then $S$ is dual equivalent to $S'$ if and only if the following is always true: 
\begin{itemize}
\item Let $T$ be a tableau whose shape extends, or is extended by, $\sh(S)$. Let $\widetilde{S}$ and $\widetilde{S'}$ be the results of shuffling $S,S'$ with $T$. Then $\sh(\widetilde{S}) = \sh(\widetilde{S}')$.
\end{itemize}
Additionally, in this case $\widetilde{S}$ and $\widetilde{S}'$ are also dual equivalent.
\end{lemma}
Thus $S$ and $S'$ are dual equivalent if \emph{their own} shapes evolve the same way under any sequence of slides. See \cite{Hai} for these and other properties of dual equivalence. Following Speyer \cite{Sp}, we extend the definition of \newword{shuffling} to dual equivalence classes:

\begin{lemma}
Let $S,T$ be skew tableaux, with $\sh(T)$ extending $\sh(S)$, and let $(S,T)$ shuffle to $(\widetilde{T},\widetilde{S})$. The dual equivalence classes of $\widetilde{T}$ and $\widetilde{S}$ depend only on the dual equivalence classes of $S$ and $T$.
\end{lemma}

The fact that rectification of skew tableaux is well-defined, regardless of the rectification order (the `fundamental theorem of JDT') is the following statement:
\begin{thm} \label{thm-dual-jdt}
Any two tableaux of the same straight shape are dual equivalent.
\end{thm}
We will write $D_\lambda$ for the unique dual equivalence class of straight shape $\lambda$. \\

Since we may use any tableau of straight shape $\beta$ to rectify a skew tableau $S$ of shape $\alpha/\beta$, we may speak of the \newword{rectification tableau} of a slide equivalence class. Similarly, by Lemma \ref{lem-dual-def2} and Theorem \ref{thm-dual-jdt} we may speak of the \newword{rectification shape of a dual equivalence class} $\rsh(D)$: this is the shape of any rectification of any representative of the class $D$. 

\begin{lemma}\label{slide-dual}
Let $D,S$ be a dual equivalence class and a slide equivalence class, with $\rsh(D) = \sh(\mathrm{rect}(S))$. There is a unique tableau in $D \cap S$.
\end{lemma}
\begin{proof}
Uniqueness is clear. To produce the tableau, pick any $T_D \in D$. Rectify $T_D$ using an arbitrary tableau $X$, so $(X,T_D)$ shuffles to $(\widetilde{T_D},\widetilde{X})$ (and $X$ and $ \widetilde{T_D}$ are of straight shape). Replace $\widetilde{T_D}$ by the rectification tableau $R_S$ for the class $S$, and let $(R_S,\widetilde{X})$ shuffle back to $(X,T)$. Then $T$ and $R_S$ are slide equivalent, and by Theorem \ref{thm-dual-jdt} and Lemma \ref{lem-dual-def2}, $T$ and $T_D$ are dual equivalent.
\end{proof}

The dual equivalence classes of a given shape and rectification shape are counted by a Littlewood-Richardson coefficient:

\begin{lemma} \label{lem-dual-LRcoeff}
Let $\beta/\alpha$ be a skew shape and let 
\[X_\alpha^\beta(\lambda) = \{\text{dual equivalence classes } D \text{ with } \sh(D) = \beta/\alpha \text{ and } \rsh(D) = \lambda \}.\]
Then $|X_\alpha^\beta(\lambda)| = c_{\alpha \lambda}^\beta.$
\end{lemma}
\begin{proof}
It is well-known that $c_{\alpha \lambda}^\beta$ counts tableaux $T$ of shape $\beta/\alpha$ whose rectification is the highest-weight tableau of shape $\lambda$. This specifies the slide equivalence class of $T$; by Lemma \ref{slide-dual}, such tableaux are in bijection with $X_\alpha^\beta(\lambda)$.
\end{proof}

We remark that tableau shuffling commutes with rotation by $180^\circ$. Let $T$ be a tableau of skew shape $\alpha/\beta$, and write $T^R$ for the tableau of shape $\beta^c/\alpha^c$ obtained by rotating $T$ by $180^\circ$, then reversing the numbering of its entries. Then the dual equivalence class of $T^R$ depends only on the dual equivalence class of $T$. This gives an involution of dual equivalence classes
\[D \mapsto D^R : X_\alpha^\beta(\lambda) \to X_{\beta^c}^{\alpha^c}(\lambda).\]
In particular, it follows that any tableaux $T, T'$ of `anti-straight-shape' $\rect/\lambda^c$ are dual equivalent, and their rectifications have shape $\lambda$.

We define a \newword{chain of dual equivalence classes} to be a sequence $(D_1, \ldots, D_r)$ of dual equivalence classes, such that $\sh(D_{i+1})$ extends $\sh(D_i)$, for each $i$. We say the chain has \newword{type} $(\lambda_1, \ldots, \lambda_r)$ if for each $i$, $\rsh(D_i) = \lambda_i$. Let $X_\alpha^\beta(\lambda_1, \ldots, \lambda_r)$ denote the set of chains of dual equivalence classes of type $(\lambda_1, \ldots, \lambda_r)$, such that $\sh(D_1)$ extends $\alpha$ and $\beta$ extends $\sh(D_r)$. This has cardinality equal to the Littlewood-Richardson coefficient $c_{\alpha, \lambda_1, \ldots, \lambda_r}^\beta$. 

Note that there is a natural identification $X_\alpha^\beta(\ybox, \cdots, \ybox)$ (with $|\beta/\alpha|$ boxes) with the set $\mathrm{SYT}(\beta/\alpha)$ of skew standard tableaux. We will think of chains of dual equivalence classes as generalizations of standard tableaux.

\subsubsection{Operations on chains of dual classes} \label{sec:shuffling-ops} We define the \newword{shuffling} operation
\[\sh_i : X_\alpha^\beta(\lambda_1, \ldots, \lambda_i, \lambda_{i+1}, \cdots \lambda_r) \to X_\alpha^\beta(\lambda_1, \ldots, \lambda_{i+1}, \lambda_i, \cdots \lambda_r)\]
by shuffling $(D_i,D_{i+1})$. These satisfy the relations $\sh_i^2 = \mathrm{id}$ and $\sh_i \sh_j = \sh_j \sh_i$ when $|i-j| > 1$. Note, however, that $\sh_i \sh_{i+1} \sh_i \ne \sh_{i+1} \sh_i \sh_{i+1}$ in general. (In the case where $\lambda_i = \ybox$ for all $i$, $\sh_i$ reduces to the Bender-Knuth involution for standard tableaux.)

We next define the $i$-th \newword{evacuation} operation 
\[\ev_i : X_\alpha^\beta(\lambda_1, \ldots, \lambda_r) \to X_\alpha^\beta(\lambda_i, \ldots, \lambda_1, \lambda_{i+1}, \ldots, \lambda_r)\]
by $\ev_i =  \sh_1 (\sh_2 \sh_1) \cdots (\sh_{i-2} \cdots \sh_1) (\sh_{i-1} \cdots \sh_1)$. This results in reversing the first $i$ parts of the chain's type, by first shuffling $D_1$ outwards past $D_i$, then shuffling the $D_2'$ (now the first element of the chain) out past $D_i'$, and so on.

In the case where $\alpha = \eset$ and $\lambda_i = \ybox$ for all $i$, the operation $\ev_i$ reduces to evacuation of the standard tableau formed by the first $i$ entries. In general, $\ev_i$ is an involution:
\begin{lemma} \label{evac-involution}
The operation $\ev_i$ is an involution.
\end{lemma}
\begin{proof}
By definition, $\ev_i = \ev_{i-1} (\sh_{i-1} \cdots \sh_1)$. On the other hand, observe that $(\sh_{i-1} \cdots \sh_1)\ev_i = \ev_{i-1}$. (Each extra $\sh_j$ cancels the leftmost instance of $\sh_j$ in $\ev_i$.) Thus we have
\[\ev_i^2 = \ev_{i-1}(\sh_{i-1} \cdots \sh_1) \ev_i = \ev_{i-1}^2,\]
and the claim follows by induction.
\end{proof}
In the case $\alpha = \eset$ and $\beta = \rect$, the operation $\ev_r$ is just reversal:
\begin{lemma} \label{DE-evac-involution}
The operation
\[\ev_r : X_\eset^{\rect}(\lambda_1, \ldots, \lambda_r) \to X_\eset^{\rect}(\lambda_r, \ldots, \lambda_1)\]
is given by $\ev_r(D_1, \ldots, D_r) = (D_r^R, \ldots, D_1^R)$.
\end{lemma}
We will give a proof below, using growth diagrams of dual equivalence classes.

Finally, we define the $i$-th \newword{evacuation-shuffle} operation
\[\esh_i : X_\eset^{\rect}(\lambda_1, \ldots, \lambda_i, \lambda_{i+1}, \cdots \lambda_r) \to X_\eset^{\rect}(\lambda_1, \ldots, \lambda_{i+1}, \lambda_i, \cdots \lambda_r)\]
by
\[\esh_i = \ev_{i+1}^{-1} \sh_1 \ev_{i+1}.\]
This operation is simpler than it appears: it only affects the $i$-th and $(i+1)$-th entries of the chain, and its effect is local. Moreover, it does not depend on the other dual equivalence classes in the chain. We have the following:
\begin{lemma} \label{upper-shuffle}
Let ${\bf D} = (D_1,\ldots, D_r) \in X_\eset^{\rect}(\lambda_1, \ldots, \lambda_r)$ and write
\[\esh_i({\bf D}) = (D_1', \ldots, D'_{i+1}, D'_i, \ldots, D_r').\]
\begin{itemize}
\item[(i)] For $j \ne i, i+1$, we have $D_j = D_j'$.
\item[(ii)] The remaining two classes $D_i', D_{i+1}'$ are computed as follows: Let $\tau, \sigma$ be, respectively, the inner shape of $D_i$ and the outer shape of $D_{i+1}$. Let ${\bf D}^* = (D_\tau, D_{i}, D_{i+1}) \in X_\eset^\sigma(\tau, \lambda_i, \lambda_{i+1})$, with $D_\tau$ the unique dual equivalence class of straight shape $\tau$. Then
\[\esh_2({\bf D}^*) = \sh_1 \sh_2 \sh_1 \sh_2 \sh_1({\bf D}^*) = (D_\tau, D_{i+1}', D_i').\]
\end{itemize}
%
\end{lemma}

We will also prove this using growth diagrams. For now, we note that from the definition, $\esh_i^2 = \mathrm{id}$, and by Lemma \ref{upper-shuffle}, when $|i-j|>1$, $\esh_i \esh_j = \esh_j \esh_i$ and $\esh_i \sh_j = \sh_j \esh_i$. \\

Let $G$ and $D$ be as in the definition of (ordinary) growth diagrams. Let $\lambda_{ij}$ be an assignment of a partition to each vertex $(i,j)$ of $D$, and let $H_{ij}, V_{ij}$ be assignments of dual equivalence classes to the horizontal and vertical edges beginning at $(i,j)$. We say the triple $\Gamma = (\lambda_{ij}, H_{ij}, V_{ij})$ is a \newword{dual equivalence growth diagram} if:
\begin{enumerate}
\item[(i)] For each directed edge $\xymatrix{\alpha \ar[r]^T & \beta}$, the dual equivalence class $T$ has shape $\beta/\alpha$,
\item[(ii)] For each square
\[\xymatrix{
\alpha \ar[r]^T & \beta \\
\gamma \ar[u]^S \ar[r]_{\widetilde{T}} & \delta, \ar[u]_{\widetilde{S}}
}\]
the dual equivalence classes $(S,T)$ shuffle to $(\widetilde{T},\widetilde{S})$.
\end{enumerate}

By definition, shuffling of dual equivalence classes is computed by choosing representatives, then computing the shuffles using an ordinary growth diagram with edges described by the square shown above. Thus each \emph{square} in a dual equivalence growth diagram is an equivalence class of ordinary growth diagrams. (A dual equivalence growth diagram in which adjacent partitions differ by one box is the same as an ordinary growth diagram.)

We will again only consider dual equivalence growth diagrams on the downwards-slanting diagonal region
\[D = \{(i,j) : 0 \leq i + j \leq r\},\]
with every vertex on the main diagonal labeled $\eset$, and every vertex on the outer diagonal labeled $\rect$. We omit the leftmost and rightmost edge labels. We call such a diagram a \newword{dual equivalence cylindrical growth diagram}, or \newword{decgd}. Decgds inherit the periodicity and symmetry of ordinary cylindrical growth diagrams:

\begin{lemma}\label{DE-cylindrical-recurrence}
Let $\Gamma = (\lambda_{ij}, H_{ij}, V_{ij})$ be a decgd. Then:
\begin{itemize}
\item[(i)] $\lambda_{(i+r)(j-r)} = \lambda_{ij}, H_{(i+r)(j-r)} = H_{ij}$, and $V_{(i+r)(j-r)} = V_{ij}$;
\item[(ii)] $\lambda_{(r-j)(-i)} = \lambda_{ij}^c, H_{(r-1-j)(-i)} = V_{ij}^R$, and $V_{(r-j)(-i-1)} = H_{ij}^R$.
\end{itemize}
\end{lemma}

\begin{proof}
Choose a fixed path across $\Gamma$ and choose tableau representatives for each dual equivalence class along the path. Consider a new diagram obtained by replacing each edge in the path by a sequence of edges encoding the chosen tableau. This extends to a unique \emph{ordinary} growth diagram $\Gamma'$, using the recurrence rule. Then the result for decgds follows from Lemma \ref{cylindrical-recurrence} for $\Gamma'$.
\end{proof}
We say the decgd has \newword{type} $(\lambda_1, \ldots, \lambda_r)$ if the entries of the first superdiagonal are the partitions $\lambda_1, \ldots, \lambda_r$. In particular, the type of the decgd is the same as the type of the chain of dual equivalence classes in its first row.

Any path from the main diagonal to the rightmost diagonal gives a chain of dual equivalence classes; on the other hand, by the recurrence condition and the uniqueness of the outermost edge labels, this uniquely specifies the remaining entries of the growth diagram. 

\begin{figure}[h]
\centering
\[\xymatrix{
\eset \ar[r] & \lambda_1 \ar[r]^-{D_2} & \cdot \ar@{..}[r] &\cdot \ar[r]^-{D_{r-1}} & \cdot  \ar[r] & \rect \\
& \eset \ar[r]\ar[u] & \lambda_2 &&&\lambda_1^c \ar[u] \ar[r] & \rect \\
&&  \ar@{}[ul]|{\ddots} \ar@{}[u]|{\vdots} & \lambda_{r-2} \ar@{}[ul]|{\ddots} &&& \ar@{}[ul]|{\ddots} \\
&&& \ar@{}[ul]|{\ddots} \eset \ar[r] \ar[u]& \lambda_{r-1} \\
&&&& \eset \ar[u] \ar[r] & \lambda_r \\
&&&&& \eset \ar[u]
}\]
\caption{The top row of this decgd is a chain of dual equivalence classes of type $(\lambda_1, \ldots, \lambda_r)$. It uniquely determines the remaining entries of the diagram.}
\label{fig:decgd}
\end{figure}
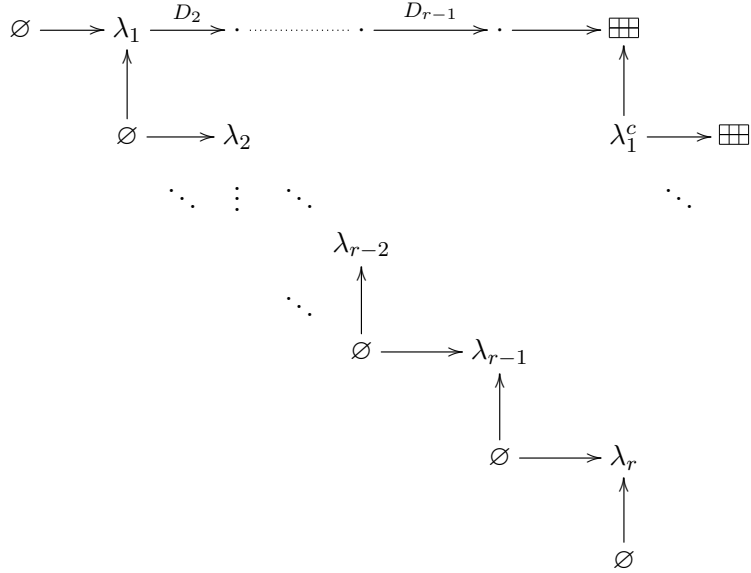

We now prove Lemmas \ref{DE-evac-involution} and \ref{upper-shuffle}. We first describe $\ev_{i+1}$  in terms of decgds. Let $\Gamma$ be the decgd whose $j=0$ row is given by ${\bf D} = (D_1, \ldots, D_r)$. By the definition of evacuation, $\ev_{i+1}({\bf D})$ is the concatenation ${\bf AB}$, where ${\bf A}$ is the chain of labels on the vertical path from $(i+1,-i-1)$ to $(i+1,0)$ and ${\bf B} = (D_{i+2}, \ldots, D_r)$ is the chain of labels on the horizontal path from $(i+1,0)$ to $(r,0)$. 

\begin{proof}[Proof of Lemma \ref{DE-evac-involution}]
Setting $i+1=r$, we have $\ev_r({\bf D}) = {\bf A}$, the path from $(r,-r)$ to $(r,0)$. By Lemma \ref{DE-cylindrical-recurrence}(ii), this sequence is $(D_r^R, \ldots, D_1^R)$.
\end{proof}

\begin{proof}[Proof of Lemma \ref{upper-shuffle}]
Let ${\bf D} = ({\bf W}, D_i, D_{i+1}, {\bf B})$ and ${\bf A} = (D_{\lambda_{i+1}},X,{\bf A'})$. We build a new decgd $\Gamma'$ as follows: we replace ${\bf A}$ by $\sh_1({\bf A})$ (in the same location) and keep ${\bf B}$ unchanged. By the recurrence condition, the remaining entries of the decgd are uniquely determined from $\sh_1({\bf A})$ and ${\bf B}$; by definition, the first row of $\Gamma'$ is $\esh_i(\bf{D})$. (See Figure \ref{fig:upper-shuffle}.) Since ${\bf B}$ and ${\bf A'}$ are unchanged in $\Gamma'$, so is ${\bf W}^R$ and therefore (by Lemma \ref{DE-cylindrical-recurrence}(ii)) ${\bf W}$. In particular, we see that ${\bf D}$ and $\esh_i({\bf D})$ agree outside the $i,i+1$ spots.

\begin{figure}[h]
\centering
\includegraphics[scale=0.7]{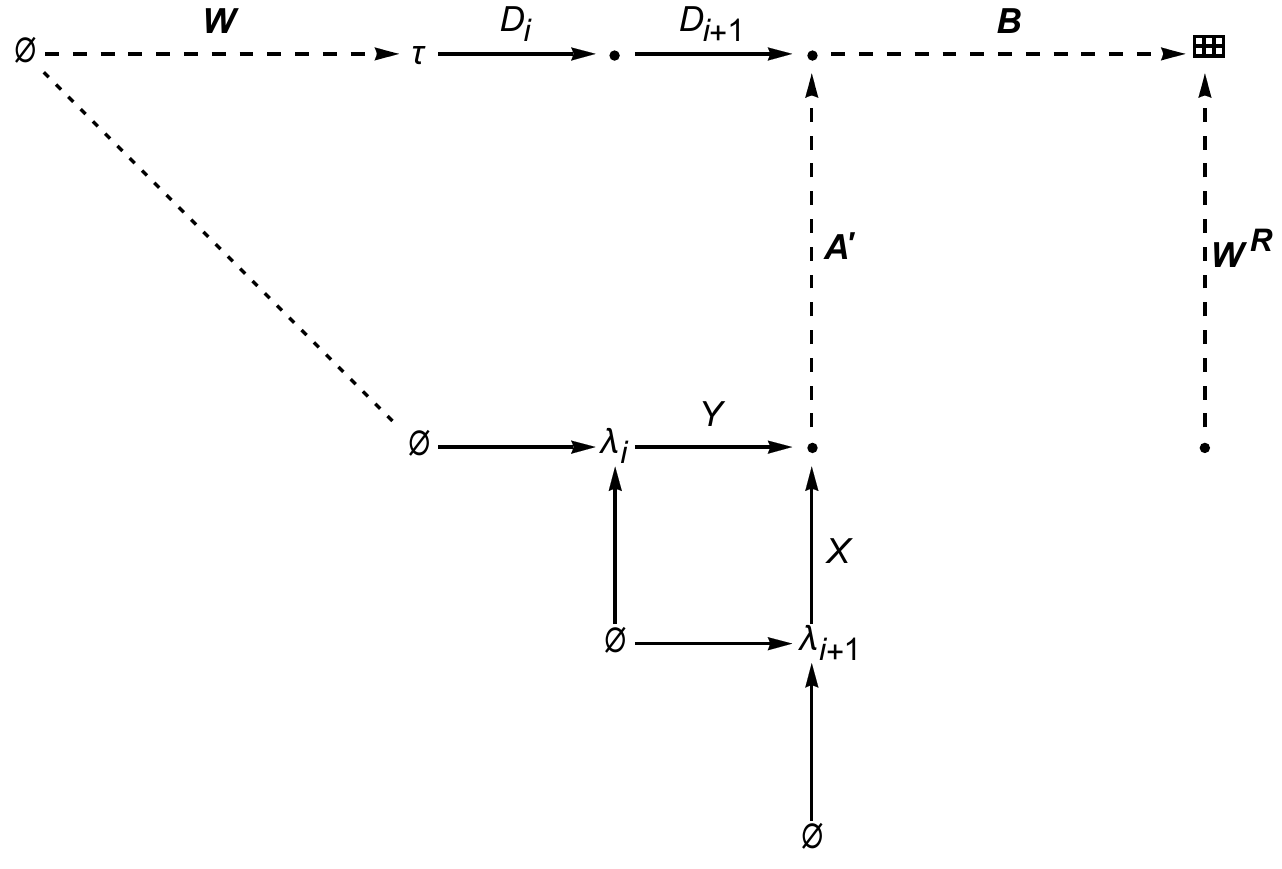} \hspace{0.5in}
\includegraphics[scale=0.7]{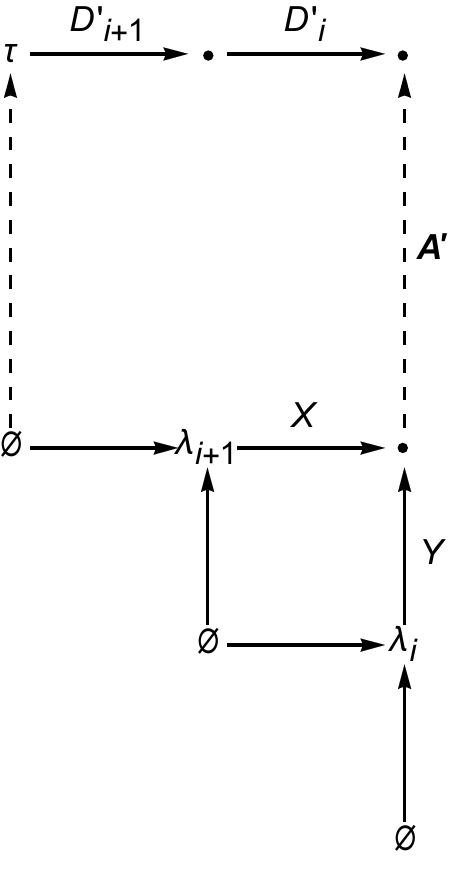}
\caption{Computing $\esh_i$ using a pair of decgds $\Gamma, \Gamma'$. Only the central portion changes. Left: the decgd $\Gamma$. Right: the central portion of the decgd $\Gamma'$.}
\label{fig:upper-shuffle}
\end{figure}

With notation as in Figure \ref{fig:upper-shuffle}, we have $\sh_1({\bf A}) = (D_{\lambda_i}, Y, {\bf A}')$. Thus the central portion of $\Gamma'$ is the second decgd pictured.
To compute $(D'_i, D'_{i+1})$, let $A'$ be the dual equivalence class obtained by concatenating the classes of the chain ${\bf A'}$ and let $\tau$ be its rectification shape. We have:
\begin{align*}
\sh_2\sh_1(D_\tau, D_i, D_{i+1}) &= (D_{\lambda_i},Y,A'), \hspace{0.5cm} \text{(from the first decgd)}\\
\sh_2\sh_1(D_\tau, D'_{i+1}, D'_i) &= (D_{\lambda_{i+1}},X,A') \hspace{0.3cm} \text{(from the second decgd)}
\end{align*}
This gives the desired relation $(D_\tau, D'_{i+1}, D'_i) = \sh_1 \sh_2 \sh_1 \sh_2 \sh_1(D_\tau, D_i, D_{i+1})$.
\end{proof}

\section{Schubert problems over $\Mbar{r}(\RR)$} \label{sec:schubert-real}

By Theorem \ref{box-lift}, we may think of $\cS(\lambda_\bullet)$ as having an extra marked point $z$, labeled by a single box, parametrizing the last point of ramification, which gives a map $\cS(\lambda_\bullet) \to \cC$. We recall our results for stable curves defined over $\RR$:
\begin{cor} \label{basic-real-topology}
Let $[C] \in \Mbar{r}(\RR)$ and let $S = \cS(\lambda_\bullet)\big|_{[C]}$ be the fiber over $[C]$. We have a finite flat map $S \to C$.
\begin{itemize}
\item[(i)] The map $S \to C$ is unramified over the real points of $C$. In particular, the only real singular points of $S$ are irreducible components meeting at simple nodes.
\item[(ii)] The map $S(\RR) \to C(\RR)$ is a covering map.
\item[(iii)] For every irreducible component $S' \subseteq S$, $S'(\RR)$ is a smooth manifold of (real) dimension 1. In particular, $S'(\RR)$ is nonempty.
\end{itemize}
\end{cor}
If $C$ has a single component, $S(\RR)$ is smooth. In particular, as $C$ varies over a maximal cell of $\Mbar{r}(\RR)$, the real topology of $S(\RR)$ (notably the number of connected components) does not change. We give a combinatorial interpretation of the connected components of $S(\RR)$ below.

We remark that $S(\CC)$ need not be connected (see Example \ref{exa:disconnected}). Also, we have not ruled out the possibility that $S$ may have complex conjugate pairs of singularities. 
%
%
We note that (if the generic fiber is smooth) a generic singular fiber of $\SLdot$ over a complex point $[C] \in \Mo{r}$ should have only one singularity. But if $[C] \in \Mo{r}(\RR)$, there must be at least two distinct singular points. We have the following conjecture:

\begin{conj}
Let $X \subseteq \Mo{r}$ be the closure of the locus where the fiber of $\SLdot$ has at least 2 singularities. Then $\mathrm{codim}(X) \geq 2$. In particular, $\Mo{r}(\RR) - X(\RR)$ is connected, so every fiber of $\SLdot$ over $\Mo{r}(\RR) - X(\RR)$ has the same \emph{complex} topology.
\end{conj}

In certain cases, there are no singularities:

\begin{exa}
Let $\lambda_\bullet = \{\ybox, \ybox, \ybox, \ybox, \ybox\}$ and consider $\SLdot \subseteq \cG(2,5) \to \Mbar{5}$. Let $[C] \in \Mo{5}(\RR)$; then the (complex) curve $S = \SLdot|_{[C]}$ is smooth. To show this, we compute in coordinates: we set $p_1, p_2, p_3, p_4, p_5$ to be $0,1,\infty,z,w$ and work over $\Mo{5} \cong \mathbb{A}^2_{z,w} - \mathbb{V}(zw(z-1)(w-1)(w-z))$.

For real $z,w$, a singular point $s \in S$ cannot satisfy a stricter Schubert condition at any marked point, since the covering $S \to \PP^1$ must send $s$ to a complex point. So we may work in the $\lambda = \ybox$ open Schubert cell for the flag $\sF(\infty)$:
\[
\left(
\begin{array}{ccccc}
 1 & a & 0 & b & c \\
 0 & 0 & 1 & d & e
\end{array}
\right) \subseteq G(2,5).\]
We may eliminate $b,c,e$ from the saturated ideal for the remaining four Schubert conditions, and are left with one equation $f_{z,w}(a,d)$, giving us a plane curve in $\mathbb{A}^2_{a,d}$. We consider the locus
\[X = \{\mathrm{disc}_a(\mathrm{disc}_d(f_{z,w})) = 0 \} \subseteq \mathbb{A}^2_{z,w}.\]
The discriminant $\mathrm{disc}_d(f)$ gives the ramification locus of $S$ under the projection $\mathbb{A}^2_{a,d} \to \mathbb{A}^1_a$; then the $a$-discriminant gives the locus where the ramification index is at least $2$. In particular, this includes any singularity, so $X$ includes any $(z,w)$ for which $S$ is a singular curve. The equation for $X$ is:
\begin{align*}
&2415919104 \big(z(z-1)w(w-1)(w-z)\big)^4\\
& (w^2-w+1)(z^2-z+1)(z^2-zw+w^2) \\
&(1 - w + w^2 - z - w z + z^2) (w^2 - 
   w z - w^2 z + z^2 - w z^2 + w^2 z^2) = 0.
\end{align*}
The factor on line 1 is a unit; the remaining factors have real solutions only at $z=w=0$ and $z=w=1$ (which are not in the open set $\Mo{5}$).
\end{exa}

\begin{rmk}
The discriminant above is a sum of squares. For example, the last two factors are
\begin{align*}
1 - w + w^2 - z - w z + z^2 &= \tfrac{1}{2}\big( (w-z)^2 + (w-1)^2 + (z-1)^2 \big), \\
w^2 - w z - w^2 z + z^2 - w z^2 + w^2 z^2 &= \tfrac{1}{2}\big( (w-z)^2 + (wz - w)^2 + (wz-z)^2\big).
\end{align*}
Sottile has conjectured that for zero-dimensional Schubert problems, the discriminants are always sums of squares of this form (see e.g. Conjecture 7.8 in \cite{Sottile}), and are in fact strictly positive.
If the same holds for one-dimensional Schubert problems, it would follow that the fibers of $\SLdot$ over $\Mo{r}(\RR)$ are smooth algebraic curves.

We also note that each quadratic factor is the pullback, by one of the five possible forgetting maps $\Mbar{5} \to \Mbar{4}$, of the pair of points on $\Mbar{4}$ having symmetry group $A_4$. It is interesting to note that the nonreduced (complex) fibers of the zero-dimensional family $\cS(\ybox^4) \subset \cG(2,4)$ over $\Mbar{4}$ also occur over this pair of points.
\end{rmk}

\begin{conj}
Let $[C] \in \Mo{r}(\RR)$. Then the (complex) fiber $\SLdot|_{[C]}$ is smooth, for any $\lambda_\bullet$.
\end{conj}
In this case the complex topology of $\SLdot|_{[C]}$ will not change over any maximal cell of $\Mo{r}(\RR)$.


\subsection{Connected components of real fibers}
We now recall Speyer's description of the topology of zero-dimensional Schubert problems $\SLdot(\RR)$, as covering spaces of $\Mbar{r}(\RR)$.

Let $X$ be a maximal cell of $\Mbar{r}(\RR)$, corresponding to a circular ordering $\sigma(1), \ldots, \sigma(r)$ of the marked points. Let $Y \subset \SLdot(\RR)$ be a cell lying over $X$. Consider an arc in $\overline{X}$ corresponding to a degeneration of $\PP^1$ to a curve $C_1 \cup C_2$, where $C_1$ contains $\sigma(i), \ldots, \sigma(j)$, and $C_2$ contains $\sigma(j+1), \ldots, \sigma(i-1)$ (in circular order). Let $S$ be the limit fiber of $\SLdot$ and $y = S \cap \overline{Y}$ the point obtained by lifting the arc to $\overline{Y}$. By Theorem \ref{thm:recall-Sp14}, $y$ corresponds to some node labeling on $C_1 \cup C_2$; we denote by $\lambda_{j,-i}$ the partition on the $C_2$ side. These partitions turn out to be organized in a dual equivalence growth diagram.

\begin{thm}[Theorem 1.6 and Proposition 7.6 of \cite{Sp}] \label{speyer-covering-space}
Let $\sum |\lambda_i| = k(n-k)$. Then $\SLdot(\RR)$ is a covering space of $\Mbar{r}(\RR)$, so we may lift the CW-complex structure of $\Mbar{r}(\RR)$ to $\SLdot(\RR)$. In particular:
\begin{itemize}
\item[(i)] Let $X$ be a maximal cell of $\Mbar{r}(\RR)$, corresponding to a circular ordering $(\sigma(1), \ldots, \sigma(r))$ of the marked points. The cells $Y$ of $\SLdot(\RR)$ lying over $X$ are indexed by decgds $\Gamma$ of type $(\lambda_{\sigma(1)}, \ldots, \lambda_{\sigma(r)})$.
\item[(ii)] (Wall-crossing) Let $Y$ be a cell lying over $X$ and $\Gamma$ the corresponding decgd. Let $X'$ be the cell obtained by reversing the interval $\sigma(i)\sigma(i+1) \cdots \sigma(j)$ in the circular ordering, and $Y'$ the corresponding cell in $\SLdot(\RR)$. Let $A$ be the triangular region of $\Gamma$ with vertices $(i,-i), (j,-j), (j,-i)$ and $B$ be the ``opposite'' triangle with vertices $(r+j,-i),(j,r-i),(j,-i)$ (see Figure \ref{fig:wall-crossing}). The decgd $\Gamma'$ for $Y'$ is obtained by transposing $A$, leaving $B$ unchanged, deleting all other entries, and refilling them using the decgd recurrence condition.
\end{itemize}
\end{thm}

\begin{figure}[h]
\centering
\includegraphics[scale=1.1]{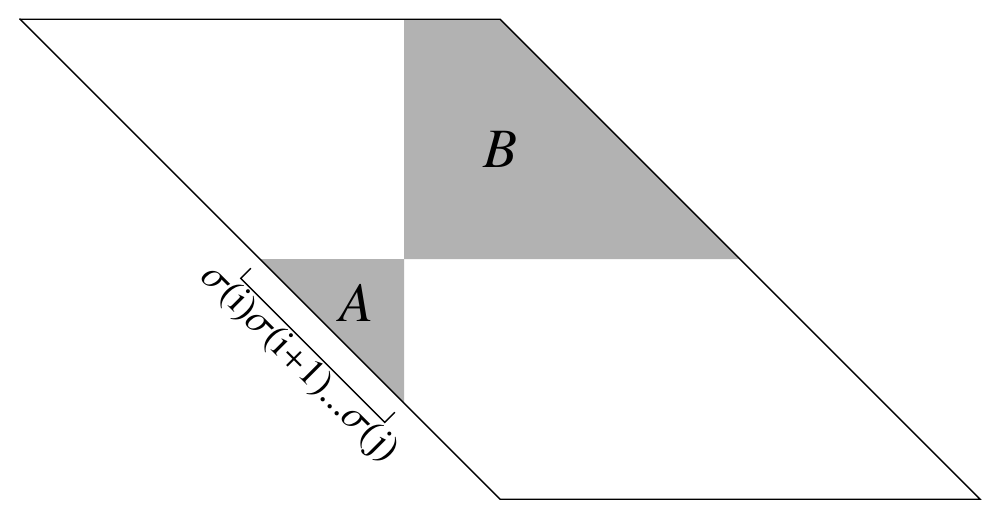}
\caption{The wall-crossing rule for decgds (shown truncated at top and bottom):
Transpose $A$ and leave $B$ unchanged; refill using the decgd recurrence rule.}
\label{fig:wall-crossing}
\end{figure}
We note that a path from the left edge to the right edge of $\Gamma$ corresponds to a choice of caterpillar curve $[\widetilde{C}]$ in the boundary of $X$. The resulting chain of partitions forms the node labeling corresponding to the point $y \in \overline{Y}$ lying over $[\widetilde{C}]$; in fact Theorem \ref{speyer-covering-space} says $y$ has the additional data of the chain of dual equivalence classes.
%
%

We now return to the case of curves. By Theorem \ref{box-lift}, when $\sum |\lambda_i| = k(n-k) - 1$, the total space of $\SLdot$ over $\Mbar{r}$ is isomorphic to the total space of $\SLdotBox$ over $\Mbar{r+1}$. Since we wish to think of this space as fibered in curves over $\Mbar{r}$, we adapt the description from Theorem \ref{speyer-covering-space}. For simplicity, we take the circular ordering $\sigma(i) = i$. Let $\mathrm{DECGD}(\ybox, \lambda_1, \ldots, \lambda_r)$ be the set of decgds of type $(\ybox, \lambda_1, \ldots, \lambda_r)$. Let 
\[\pi : \mathrm{DECGD}(\ybox, \lambda_1, \ldots, \lambda_r) \to \mathrm{DECGD}(\ybox, \lambda_1, \ldots, \lambda_r)\]
be the result of successively wall-crossing $\ybox$ past each of the $\lambda_i$'s ($i=1, \ldots, r$).
\begin{thm} \label{thm:decgd-curve-orbits}
Let $\sum |\lambda_i| = k(n-k)-1$. Let $X$ be the maximal cell of $\Mo{r}(\RR)$ corresponding to the circular ordering $1, 2, \ldots, r$, and let $S = \SLdot|_X$. The connected components of $S(\RR)$ are in bijection with the orbits of $\pi$; each component is homeomorphic to $S^1 \times X$.
\end{thm}
\begin{proof}
Let $[C] \in X$. The fiber $C \subseteq \Mbar{r+1}$ passes through $r$ maximal cells of $\Mbar{r+1}$, corresponding to the possible placements of the $(r+1)$-st marked point. The decgds labeling these cells for the covering space $\SLdotBox(\RR) \to \Mbar{r+1}(\RR)$ have type $(\lambda_1, \ldots, \lambda_i, \ybox, \lambda_{i+1}, \ldots, \lambda_r)$. When the $\ybox$ switches places with $p_i$, we apply the wall-crossing procedure.

Thus, when $\ybox$ travels around the $\RR\PP^1$, the decgd changes by $\pi$. Since $S(\RR)|_{[C]}$ is a union of circles and the topology does not change as $[C]$ varies over $X$, the homeomorphism follows.
\end{proof}

A natural question is whether $S(\RR)$ has the same number of connected components over every cell $X$. We address this and related questions in the next section.

\subsection{Caterpillar curves and desingularizations}
We give a different combinatorial description with two advantages: first, it is more amenable to computation; second, it makes it easier to compare $\SLdot$ over different cells of $\Mo{r}(\RR)$. It will also connect the operator $\pi$ of Theorem \ref{thm:decgd-curve-orbits} to promotion and evacuation of tableaux.

The idea is to pass to a caterpillar curve $\widetilde{C}$ in the boundary of the maximal cell. We describe the covering space $\widetilde{S}(\RR) \to \widetilde{C}(\RR)$ in terms of chains of dual equivalence classes. For the remainder of this section, let $\widetilde{C}$ be the caterpillar curve with marked points, from left to right, $p_1, \ldots, p_r$. Let the nodes be $q_{23}, \ldots, q_{(r-2)(r-1)}$, and let $q_{12} = p_1$ and $q_{(r-1)r} = p_r$. For $i = 2, \ldots, r-1$, let $\ell_i$ be the arc from $q_{(i-1)i}$ to $q_{i(i+1)}$ \emph{through} $p_i$, and let $u_i$ be the arc \emph{opposite} $p_i$.

We define a covering space $S_{DE} \to \widetilde{C}(\RR)$ as follows:
\begin{enumerate}
\item[(i)] The fiber of $S_{DE}$ over $q_{i(i+1)}$ is indexed by the set $X_\eset^{\rect}(\lambda_1, \ldots, \lambda_i,\ybox,\lambda_{i+1}, \ldots, \lambda_r)$.
\item[(ii)] The arcs covering $\ell_i$ connect ${\bf D}$ to $\esh_i({\bf D})$, where
\[\esh_i : X_\eset^{\rect}(\lambda_1, \ldots, \ybox,\lambda_i, \ldots, \lambda_r) \to X_\eset^{\rect}(\lambda_1, \ldots, \lambda_i,\ybox, \ldots, \lambda_r)\] is the $i$-th evacuation-shuffle.
\item[(iii)] The arcs covering $u_i$ connect ${\bf D}$ to $\sh_i({\bf D})$, where
\[\sh_i : X_\eset^{\rect}(\lambda_1, \ldots, \ybox,\lambda_i, \ldots, \lambda_r) \to X_\eset^{\rect}(\lambda_1, \ldots, \lambda_i,\ybox, \ldots, \lambda_r)\] is the $i$-th shuffle.
\end{enumerate}
(Note: we do not explicitly label the fibers over $p_2, \ldots, p_{r-1}$.) \\

See Figure \ref{fig:covering-space} for a possible such covering, along with the smooth curves obtained by desingularizing the caterpillar curve.

\begin{thm} \label{thm:DE-caterpillar-covering}
Let $\tilde S = \SLdot|_{[\widetilde{C}]}$. Then $S_{DE} \cong \widetilde{S}(\RR)$ as covering spaces of $\widetilde{C}(\RR)$.
\end{thm}
\begin{proof}
Let $X$ be the cell of $\Mo{r}(\RR)$ containing $\widetilde{C}$ in its boundary, corresponding to the circular ordering $12\cdots r$. For $i = 1, \ldots, r-1$, let $X_{i(i+1)}$ be the cell of $\Mo{r+1}(\RR)$ over $X$ in which $\ybox$ is between $\lambda_i$ and $\lambda_{i+1}$. Finally, let $X_{r1}$ be the cell in which $\ybox$ is between $\lambda_r$ and $\lambda_1$.

We first describe the indexing of the fiber over $q_{i(i+1)}$. Let $s \in \pi^{-1}(q_{i(i+1)})$. There are many cells of $\Mo{r+1}(\RR)$ containing $q_{i(i+1)}$ in their boundary (for example, $X_{i(i+1)}$ and $X_{r1}$). Let $X^*$ be any such cell and $Y$ the unique cell lying over $X^*$ containing $s$ in its boundary. Let $\Gamma$ be the decgd corresponding to $Y$. There is a unique path through $\Gamma$ that yields a chain of dual equivalence classes ${\bf D}$ of type $(\lambda_1, \ldots, \lambda_i, \ybox, \lambda_{i+1}, \ldots, \lambda_r)$. We label $s$ by ${\bf D}$. (By Theorem \ref{speyer-covering-space}, ${\bf D}$ does not depend on our choice of cell.) This gives (i).

We next compute the effect of lifting an arc from $q_{(i-1)i}$ to $q_{i(i+1)}$, starting from $s$.

Let $Y_{(i-1)i}$ be the cell covering $X_{(i-1)i}$ corresponding to the decgd $\Gamma$ whose first row is ${\bf D}$. Let $Y_{i(i+1)}$ be the cell obtained by following the arc $\ell_i$ (crossing a wall when $\ybox$ collides with $\lambda_i$). Let the decgd for $Y_{i(i+1)}$ be $\Gamma'$. From the wall-crossing rule of Theorem \ref{speyer-covering-space}, $\Gamma'$ is obtained by transposing the portion of $\Gamma$ consisting of $\ybox, \lambda_i$ and the partition covering the two. By Lemma \ref{upper-shuffle}, the $\eset \longrightarrow \lambda_1 \longrightarrow \cdots$ row of $\Gamma'$ is $\esh_i({\bf D})$. This gives (ii).

Now let $Y_{r1}$ be the unique cell covering $X_{r1}$ containing $s$ in its boundary, and let $\Gamma$ be its decgd. Let ${\bf D}'$ be the label on the point obtained by lifting the arc $u_i$ to $s$. The lift of $u_i$ does not cross a wall (it lies entirely in $\overline{Y_{r1}}$), so ${\bf D, D'}$ both appear in $\Gamma$ as the paths: 
\[\xymatrix{
\eset \ar[r] & \ybox \ar@{..}[r]  &\ar@{..}[r] &\ar@{..}[r] &\ar@{..}[r] & \cdot \ar[r]^{D_i} & \cdot \ar@{--}[r] &\cdot \ar[r]^-{D_{r-1}} & \lambda_r^c  \ar[r] & \rect \\
& \eset \ar[u] \ar[r] & \lambda_1 \ar[r]^-{D_2} & \cdot \ar@{--}[r] &\cdot \ar[r]^-{D_{i-1}} & \cdot \ar[u]^{D_\ybox} \ar[r]^{D_i'} & \cdot \ar[u]_{D'_\ybox} \ar@{..}[r]  & \ar@{..}[r]  & \ar@{..}[r] & \ybox^c\ar[u] \ar[r] & \rect,
}\]
so we see that ${\bf D'} = \sh_i({\bf D})$, which is (iii).
\end{proof}

We next compare $\widetilde{\pi} : \widetilde{S}(\RR) \to \widetilde{C}(\RR)$ to a nearby desingularization $\pi : S(\RR) \to C(\RR) \cong \RR\PP^1$, with $C \in \Mo{r}(\RR)$. Let $\gamma$ be the loop around the circle $C(\RR)$, starting from $p_1$ and traversing $p_2$ last. Inside $\Mbar{r+1}$, $\gamma$ is homotopic to a unique sequence of arcs around $\widetilde{C}(\RR)$, as in Figure \ref{fig:desings}. Let $\omega$ be the corresponding composition of shuffles and evacuation-shuffles. The monodromy action of $\pi_1(C(\RR))$ on $\pi^{-1}(p_1)$ is equivalent to the action of $\omega$ on $\widetilde{\pi}^{-1}(p_1) \subset \widetilde{S}(\RR)$.

\begin{figure}[h]
\centering
\includegraphics[scale=0.6]{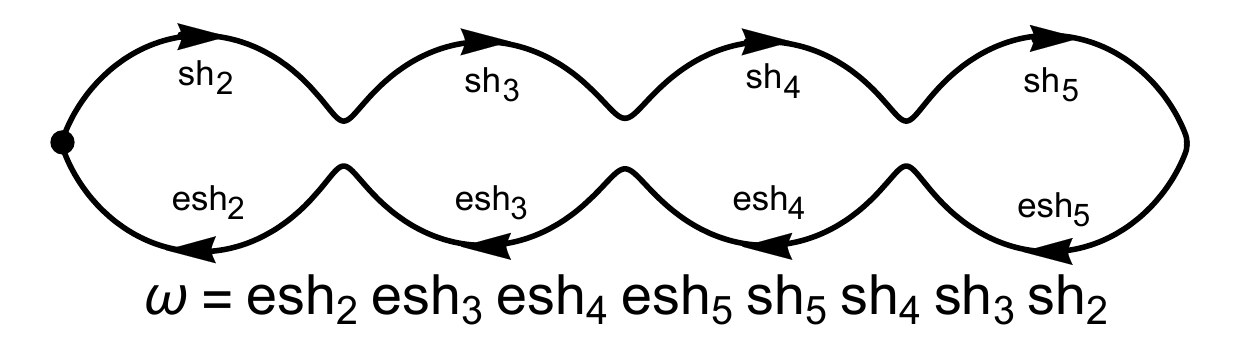}
\includegraphics[scale=0.6]{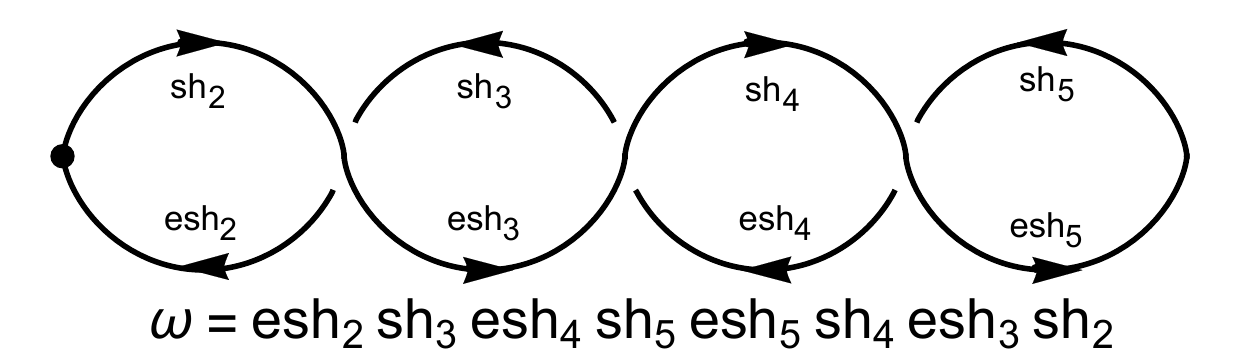}
\caption{Two desingularizations of a caterpillar curve, with the associated operations $\omega$. For $i=2, \ldots, 5$, the marked point $p_i$ is on the arc labeled $\esh_i$.}
\label{fig:desings}
\end{figure}

It is convenient to reindex the fiber of $\widetilde{S}$ over $p_1$ by $X_{\tinybox}^{\rect}(\lambda_1, \ldots, \lambda_r)$. Note that there is a canonical bijection
\[\iota : X_{\tinybox}^{\rect}(\lambda_1, \ldots, \lambda_r) \to X_\eset^{\rect}(\ybox,\lambda_1, \ldots, \lambda_r),\]
and that the two operations
\[\sh_1, \esh_1 : X_\eset^{\rect}(\ybox, \lambda_1, \ldots, \lambda_r) \to X_\eset^{\rect}(\lambda_1, \ybox, \ldots, \lambda_r)\]
are the same. We deduce:

\begin{cor} \label{cor:connected-components-desing}
Let $X$ be a maximal cell of $\Mo{r}(\RR)$, containing $[\widetilde{C}]$ in its boundary. Let $[C] \in X$ be any desingularization and let $S = S(\lambda_\bullet)|_{[C]}$ be the Schubert curve over $[C]$.
Let
\[\omega_X : X_{\tinybox}^{\rect}(\lambda_1, \ldots, \lambda_r) \to X_{\tinybox}^{\rect}(\lambda_1, \ldots, \lambda_r)\]
be the composition of shuffles and evacuation-shuffles corresponding to the loop around $C(\RR)$. There is a bijection
\[\left\{\begin{split}\text{components\ } \\ \text{ of } S(\RR) \hspace{0.4cm} \end{split}\right\} \longleftrightarrow X_{\tinybox}^{\rect}(\lambda_1, \ldots, \lambda_r) / \omega_X.\]
\end{cor}

\begin{cor} \label{cor:promotion}
If $X$ is the cell corresponding to the circular ordering $p_1, \ldots, p_r$, the connected components of $\SLdot(\RR)|_X$ are the orbits of
\[\omega_X = \iota^{-1} \circ \esh_1 \cdots \esh_{r-1} \sh_{r-1} \cdots \sh_{1} \circ \iota.\]
In particular, the connected components of $\cS(\ybox, \ldots, \ybox)(\RR)|_X$ are in bijection with the orbits of tableau promotion on $\mathrm{SYT}(\rect)$.
\end{cor}
\begin{proof}
Tableau promotion is the composition $\sh_{r-1} \cdots \sh_1$. (Recall that under the identification of $X_\tinybox^\rect(\ybox, \ldots, \ybox)$ with $SYT(\rect)$, $\esh_i$ becomes trivial.)
\end{proof}

We remark that the identification above was contingent on the choice of caterpillar curve $\widetilde{C}$ in the boundary of $X$. (The statement in Theorem \ref{thm:decgd-curve-orbits} is canonical for the entire cell, though the connection to tableau promotion is less apparent.) A different caterpillar curve $\widetilde{C}'$ in the boundary of $X$, with points $p_{\sigma(1)}, \ldots, p_{\sigma(r)}$ from left to right, yields an operator $\omega'$ that differs from $\omega$ by an intertwining operator
\[\psi : X_{\tinybox}^{\rect}(\lambda_1, \ldots, \lambda_r) \to X_\tinybox^{\rect}(\lambda_{\sigma(1)}, \ldots, \lambda_{\sigma(r)}).\]
The function $\psi$ is a sequence of shuffles and evacuation-shuffles, corresponding to changing paths in a decgd. We do not describe $\psi$ explicitly.

The advantage of Corollary \ref{cor:connected-components-desing} is that we may compare different cells $X_i$ by desingularizing $\widetilde{C}$ in different ways. We have:

\begin{cor}
Let $\eta(X)$ be the number of connected components of $\SLdot(\RR)|_X$. For any two maximal cells   $X, X'$ of $\Mo{r}(\RR)$, $\eta(X) \equiv \eta(X')\ \mathrm{mod} \ 2.$
\end{cor}
\begin{proof}
We may assume $X$ and $X'$ share a wall and that $\widetilde{C}$ is in the closure of this wall. Then the operations $\omega_X, \omega_{X'}$ are reorderings of the same set of bijections (each $\esh_i$ and $\sh_i$ appears once). Thus, as permutations of $X_{\tinybox}^{\rect}(\lambda_1, \ldots, \lambda_r)$, $\omega_X$ and $\omega_{X'}$ have the same sign. This determines the parity of the number of orbits.
\end{proof}
In general, $\eta(X)$ and $\eta(X')$ need not be equal, as the following example shows:

\begin{exa} \label{exa:minimal-counterex}
Let $\lambda_\bullet = \big\{{\tiny \yng(2), \yng(2,1), \yng(3,1), \yng(3,2) }\big\}$ and consider $\SLdot \subseteq \cG(3,8)$ over $\Mbar{4}$. Let $X, X', X'' \subseteq \Mo{4}(\RR)$ be the cells corresponding to the circular orderings $1234, 1243,1324$. Then $\eta(X) = 3$, but $\eta(X') = \eta(X'') = 1$.
\end{exa}

The absence of smaller examples is explained in part by the following.

\begin{lemma} \label{lem:mult-free-transpositions}
Let the circular orderings for $X, X'$ differ by exchanging two adjacent points $p_i, p_j$, and suppose the product $\lambda_i \cdot \lambda_j$ in $H^*(G(k,n))$ is multiplicity-free, that is, $c_{\lambda_i \lambda_j}^\nu \leq 1$ for all $\nu$. Then $\eta(X) = \eta(X')$.
\end{lemma}
\begin{proof}
We may assume $i=1$, $j=2$ and the circular ordering for $X$ is $123\cdots r$. Let $\omega_X, \omega_{X'}$ be the bijections as in Corollary \ref{cor:connected-components-desing} corresponding to the loops for $X,X'$:
\begin{align*}
\omega_X &= \esh_2 \circ \esh_3 \circ \cdots \circ \esh_{r-1} \circ \sh_{r-1} \circ \cdots \circ \sh_3 \circ \sh_2, \\
\omega_{X'} &= \sh_2 \circ \esh_3 \circ \cdots \circ \esh_{r-1} \circ \sh_{r-1} \circ \cdots \circ \sh_3 \circ \esh_2.
\end{align*}
We see that $\omega_X$ is conjugate to $\omega_{X'} (\esh_2 \sh_2)^2$, and $\esh_2 \sh_2$ corresponds to the loop around the first component of the caterpillar curve. 
 
Let $s \in \pi^{-1}(p_1) \subset \widetilde{S}(\RR)$ and ${\bf D}$ the corresponding chain of dual equivalence classes. Let $\nu$ be the node labeling with $s \in S_\nu$ and $\nu(q_{23}, C_1)$ the node label of $q_{23}$ on the first component. Since $\esh_2 \sh_2$ only affects the $\ybox, \lambda_1, \lambda_2$ dual equivalence classes, we truncate ${\bf D}$ and work in the set $X_\eset^{\rect}(\ybox,\lambda_1, \lambda_2,\nu(q_{23},C_1))$.

By the Pieri rule and our assumption on $\lambda_1$ and $\lambda_2$, $X_\eset^{\rect}(\ybox,\lambda_1, \lambda_2,\nu(q_{23},C_1))$ has cardinality $\leq 2$, so $(\esh_2 \sh_2)^2 = \mathrm{id}$. This holds for all points $s$, so $\omega_X$ and $\omega_{X'}$ are conjugate, hence have the same orbit structure.
\end{proof}
\begin{cor}\label{cor:mult-free}
Suppose every pairwise product $\lambda_i \cdot \lambda_j$ in $H^*(G(k,n))$ is multiplicity-free. Then the operators $\omega$ for different circular orderings are all conjugate. In particular, the number of real connected components of $S(\RR)$ does not depend on the ordering of the $p_\bullet$.
\end{cor}

As an example, if $\alpha, \beta$ are rectangular partitions, then $\alpha \cdot \beta$ is known to be multiplicity free. We may make a slightly stronger statement:
\begin{cor}
Let the circular orderings of $X, X'$ differ by any permutation $\sigma$, where $\sigma$ fixes all non-rectangular partitions and does not move any partition past a non-rectangular partition. Then $\eta(X) = \eta(X')$.
\end{cor}
\begin{proof}
$X$ and $X'$ are connected by a sequence of transposition wall-crossings as in Lemma \ref{lem:mult-free-transpositions}.
\end{proof}
\begin{cor}
If all or all but one $\lambda_i$ are rectangular, $\eta(X)$ is the same for all $X$.
\end{cor}
We also note that Corollary \ref{cor:mult-free} applies to any Schubert problem on $G(2,n)$. Certain other cases also trivially have $\eta(X) = \eta(X')$, such as when two identical partitions switch places. It is interesting to point out a smaller candidate counterexample $\lambda_\bullet = \big\{{\tiny \yng(2), \yng(2), \yng(2,1), \yng(3,1)}\big\}$, with $\SLdot \subset \cG(3,7)$. Here, $\eta(X) = 2$ for all circular orderings, but the permutations $\omega_X$ and $\omega_{X'}$ for the circular orderings $1234$ and $1324$ are not conjugate: $X_{\tinybox}^{\rect}(\lambda_\bullet)$ has $8$ elements, which are partitioned into two orbits of sizes $3,5$ by $\omega_X$ and $4,4$ by $\omega_{X'}$.

\section{Connections to K-theory} \label{sec:k-theory}

\subsection{Basic facts}
The classes of the Schubert structure sheaves $[\cO_\lambda] := [\cO_{\Omega(\lambda)}]$ form an additive basis for the K-theory of $G(k,n)$. We write $k_{\lambda_\bullet}^\nu$ for the absolute value of the coefficient of $[\cO_\nu]$ in the product $\prod_i [\cO_{\lambda_i}]$. This is zero unless $|\nu| \geq \sum |\lambda_i|$, and the leading terms agree with cohomology:
\[c_{\lambda_\bullet}^\nu = k_{\lambda_\bullet}^\nu \text{ when } |\nu| = \sum |\lambda_i|.\]
The coefficients alternate in sign:
\begin{thm}\emph{\cite{Bu02}} The structure constant $k_{\lambda_\bullet}^\nu$ appears with sign $\displaystyle{(-1)^{|\nu| - \sum |\lambda_i|}}.$
\end{thm}
\noindent We note that a Schubert variety for $\ybox^c$ is isomorphic to $\PP^1$; in particular, the Euler characteristic is $\chi(\cO_{\tinybox^c}) = 1$.
\begin{lemma} \label{lem:complementary-ktheory}
Suppose $|\mu| + |\lambda| = k(n-k) - 1$. Then $k_{\lambda \mu}^{\rect} = 0$ and $[\cO_\lambda] \cdot [\cO_\mu] = k_{\lambda \mu}^{\tinybox^c} [\cO_{\tinybox^c}].$
\end{lemma}
\begin{proof}
We may write
\[[\cO_\lambda] \cdot [\cO_\mu] = k_{\lambda \mu}^{\tinybox^c} [\cO_{\tinybox^c}] - k_{\lambda \mu}^{\rect} [\cO_{\rect}], \text{ and so } \chi(\cO_S) = k_{\lambda \mu}^{\tinybox^c} - k_{\lambda \mu}^{\rect},\]
where $S$ is the corresponding intersection of Schubert varieties. There are two cases. If $\mu^c \not\supset \lambda$, $S$ is empty and both coefficients are zero. Otherwise, $S$ is a reduced curve, whose degree in the Pl\"{u}cker embedding is 1 because (by the Pieri rule) $k_{\lambda \mu}^{\tinybox^c} = 1$. Hence $S$ must be isomorphic to $\PP^1$ and have Euler characteristic 1.
\end{proof}

\begin{lemma}
Let $\alpha, \beta, \gamma$ be partitions such that $|\alpha| + |\beta| + |\gamma| = k(n-k) - 1$. Then
\[[\cO_\alpha] \cdot [\cO_\beta] \cdot [\cO_\gamma] = k_{\alpha \tinybox \beta \gamma}^{\rect} [\cO_{\tinybox^c}] - k_{\alpha \beta}^{\gamma^c} [\cO_{\rect}].\]
\end{lemma}
\begin{proof}
We note that $k_{\alpha \beta \gamma}^{\tinybox^c} = k_{\alpha \tinybox \beta \gamma}^{\rect}$ by the Pieri rule (in cohomology). For the coefficient of $[\cO_{\rect}]$, by definition, we have
\[k_{\alpha \beta \gamma}^{\rect} = \sum_\nu \pm k_{\alpha \beta}^\nu k_{\nu \gamma}^{\rect}.\]
If $|\nu| + |\gamma| = k(n-k) - 1$, then $k_{\nu \gamma}^{\rect} = 0$ by Lemma \ref{lem:complementary-ktheory}. But if $|\nu| + |\gamma| = k(n-k)$, we know from cohomology that $k_{\nu \gamma}^{\rect}$ is $1$ if $\nu = \gamma^c$ and $0$ otherwise.
\end{proof}

The coefficient $k_{\alpha \beta}^{\gamma^c}$ counts increasing tableaux of shape $\gamma^c/\alpha$ whose rectification, under K-theoretic jeu de taquin, is the highest-weight standard tableau of shape $\beta$ (see \cite{ThYo}). When $|\gamma^c / \alpha| = |\beta| + 1$, any such tableau is standard except for a single repeated entry.

\subsection{Schubert curves in K-theory}

Our key connection to K-theory comes from the following:

\begin{lemma}\label{lem:euler-char-integral}
Let $S$ be a smooth, integral projective curve, defined over $\RR$, and suppose $S(\CC) - S(\RR)$ is disconnected. Let $\eta(S)$ be the number of connected components of $S(\RR)$. Then $\eta(S) \equiv \chi(\cO_S)\ (\mathrm{mod}\ 2).$
\end{lemma}
\begin{proof}
This is well-known (see, for example, \cite{GrHa}).
\end{proof}
Our curves may not be smooth or integral, but the identity holds nonetheless.
\begin{lemma}\label{lem:recall-euler-char}
Let $S = S(\lambda_\bullet,p_\bullet)$ be the Schubert curve, with $p_i \in \RR\PP^1$ for each $i$. Let $\eta(S)$ be the number of connected components of $S(\RR)$. Then $\eta(S) \equiv \chi(\cO_S)\ (\mathrm{mod}\ 2).$
\end{lemma}
\begin{proof}
Let $S$ have irreducible components $S_i$, and let $\widetilde{S} = \bigsqcup \widetilde{S_i}$, where $\widetilde{S_i} \to S_i$ is the normalization. We have a birational morphism $\pi: \widetilde{S} \to S$, and an exact sequence
\[0 \to \cO_S \to \pi_* \cO_{\widetilde{S}} \to \cF \to 0,\]
with cokernel supported at the singular points of $S$. By Theorem \ref{cor:as-real-as-poss}, $S$ has smooth real points. The singularities of $S$ therefore occur in (isomorphic) complex conjugate pairs, so $\chi(\cF) = \dim_\CC H^0(\cF)$ is even and $\chi(\cO_S) \equiv \chi(\cO_{\widetilde{S}})\ \mathrm{mod}\ 2$. By Corollary \ref{cor:normalization-disconnected}, each $\widetilde{S_i}$ is disconnected by its real points, so our conclusion follows by summing over the $\widetilde{S_i}$.
\end{proof}
We also have the following inequality:
\begin{lemma}
With notation as above, let $\iota(S)$ be the number of irreducible components of $S$. Then $\chi(\cO_S) \leq \iota(S) \leq \eta(S)$.
\end{lemma}
\begin{proof}
Since $S$ is reduced, $h^0 = \dim_\CC H^0(\cO_S)$ is the number of connected components of $S(\CC)$. We have $\chi(\cO_S) \leq h^0$. We have shown (Corollary \ref{cor:normalization-disconnected}) that every irreducible component of $S$ contains a real point, and $S(\RR)$ is smooth, so $h^0 \leq \iota(S) \leq \eta(S)$.
\end{proof}

For the remainder of this section, we specialize to the case of three partitions $\alpha, \beta, \gamma$ whose sizes sum to $k(n-k)-1$. By Corollary \ref{cor:connected-components-desing}, the connected components of $S(\RR)$ are in bijection with the orbits of $\omega = \esh_2 \circ \sh_2$, where
\[\esh_2, \sh_2 : X_\eset^{\rect}(\alpha, \ybox, \beta, \gamma) \to X_\eset^{\rect}(\alpha, \beta, \ybox, \gamma)\]
are the shuffle and evacuation-shuffle on chains of dual equivalence classes. Note that the cardinality of $X_\eset^{\rect}(\alpha, \ybox, \beta, \gamma)$ is $k_{\alpha \tinybox \beta \gamma}^{\rect}$. We have proven the following combinatorial facts:
\begin{cor} \label{cor:parity-eqn} We have
\begin{equation} \label{eqn:recall-parity-eqn}
\#\mathrm{orbits}(\omega) \equiv k_{\alpha \tinybox \beta \gamma}^{\rect} - k_{\alpha \beta}^{\gamma^c} \ (\mathrm{mod}\ 2) \ \ \text{ and }\ \  \mathrm{sign}(\omega) \equiv k_{\alpha \beta}^{\gamma^c}\ (\mathrm{mod}\ 2),
\end{equation}
where $\mathrm{sign}(\omega) = 0$ or $1$, and the inequality
\begin{equation} \label{eqn:ktheory-inequality}
k_{\alpha \tinybox \beta \gamma}^{\rect} \leq \#\mathrm{orbits}(\omega) + k_{\alpha \beta}^{\gamma^c}.
\end{equation}
\end{cor}
We note that if $k_{\alpha \beta}^{\gamma^c} = 0$, then $\omega$ is the identity permutation. In this case $[\cO_S] = k_{\alpha \tinybox \beta \gamma}^{\rect} \cdot [\cO_{\tinybox^c}]$ in K-theory, and it is easy to see that $S$ must then be a disjoint union of $k_{\alpha \tinybox \beta \gamma}^{\rect}$ copies of $\PP^1$.

\begin{exa}[A disconnected Schubert curve] \label{exa:disconnected}
Let $\alpha = \beta = \gamma = {\tiny \yng(3,1,1)}$, and let $S = S(\alpha, \beta, \gamma; p_\bullet) \subseteq G(4,8)$. Then $k_{\alpha \beta}^{\gamma^c} = 0$ and $k_{\alpha \beta \tinybox \gamma}^{\rect} = 2$, so $S \cong \PP^1 \sqcup \PP^1$.
\end{exa}

On the other hand, there are examples where $S$ is integral and $\eta(S) < g(S) + 1$:

\begin{exa}[A Schubert curve with fewer than $g+1$ components]\label{exa:large-k}
Let $S = S(\alpha, \beta, \gamma; p_\bullet) \subseteq G(4,9)$, with
\[\alpha = \gamma = {\tiny \yng(3,2,1)} \text{ and } \beta = {\tiny \yng(4,2,1)}.\]
Then $\eta(S) = 1$, $k_{\alpha \beta \tinybox \gamma}^{\rect} = 12$ and $k_{\alpha \beta}^{\gamma^c} = 13$, so $S$ is integral with arithmetic genus $2$. A computation in coordinates shows that $S$ is smooth.
\end{exa}

We do not know a combinatorial explanation in general for equations \eqref{eqn:recall-parity-eqn} or \eqref{eqn:ktheory-inequality} or their analogs for products of more than three partitions. Below, we prove equation \eqref{eqn:recall-parity-eqn} in the case where $\beta$ is a horizontal or vertical strip (the `Pieri case') and $\alpha, \gamma$ are arbitrary. By the associativity of Littlewood-Richardson numbers, this gives an independent combinatorial proof of the analog of equation \eqref{eqn:recall-parity-eqn} for arbitrary products of horizontal and vertical strips.

\begin{rmk}
In the Pieri case, \eqref{eqn:ktheory-inequality} is actually an equality, and \eqref{eqn:recall-parity-eqn} holds over $\ZZ$. Example \ref{exa:large-k} shows that this is not the case in general.
\end{rmk}

We also give a simple proof of the parity identity for the product of $k(n-k)-1$ copies of $\ybox$ (the `promotion case').

\subsection{The Pieri Case}
Let $\beta$ be a horizontal strip of length $d$ and $\alpha, \gamma$ be arbitrary. (The proof for vertical strips is entirely analogous.) Assume $c_{\alpha \tinybox \beta \gamma}^{\rect} \ne 0$. There are two cases to consider:

\emph{Case 1}. Suppose $\gamma^c/\alpha$ is not a horizontal strip. Then $\gamma^c/\alpha$ must contain a single vertical domino ${\tiny \yng(1,1)}$, but be a horizontal strip otherwise. Then $X_{\eset}^{\rect}(\alpha, \ybox, \beta, \gamma)$ has only one element, since the $\ybox$ must go in the top box of the domino, and  $k_{\alpha \beta}^{\gamma^c} = 0$. \\

\emph{Case 2}. Suppose $\gamma^c/\alpha$ is a horizontal strip of $d+1$ boxes; let $r$ be the number of nonempty rows of the skew shape $\gamma^c/\alpha$. Then there is a natural ordering of the chains
\[X_{\eset}^{\rect}(\alpha, \ybox, \beta, \gamma) = \{{\bf D}_1, \ldots, {\bf D}_r\},\]
where ${\bf D}_i$ is the chain where the $\ybox$ is at the start of the $i$-th lowest row of $\gamma^c/\alpha$. (The other dual equivalence classes are all determined by this choice.)

\begin{thm} \label{thm:pieri-case}
Let $\omega = \esh_2 \circ \sh_2$. Then $\omega({\bf D}_i) = {\bf D}_{i+1 \ \mathrm{mod}\ r}.$
\end{thm}
\begin{proof}
We first show that $\omega({\bf D}_r) = {\bf D}_1$. Observe that $\sh_2({\bf D}_r)$ has the $\ybox$ at the end of the top row of $\gamma^c/\alpha$. We think of the filling of $\gamma^c/\alpha$ as a single skew tableau $T$, with $\ybox$ as its largest entry. Then $\sh_2\sh_1(\sh_2{\bf D}_r)$ rectifies $T$, and since the entries of $T$ strictly increase from left to right, the rectification is a horizontal strip of length $d+1$, with $\ybox$ at the end. Then $\sh_1$ slides the $\ybox$ to the beginning of the strip, so $\sh_1\sh_2$ must move the $\ybox$ to the leftmost space of $\gamma^c/\alpha$, i.e. the beginning of the lowest row. (See Figure \ref{fig:pieri-r1}.) Thus $\omega({\bf D}_r) = {\bf D}_1$.

Next, we show that, for all $i$, $\omega({\bf D}_i) = {\bf D}_j$ with $j \leq i+1$. Since we know $\omega({\bf D}_r) = {\bf D}_1$, this forces $\omega$ to be the desired permutation.

We may assume $i + 1 < r$. By definition, $\sh_2({\bf D}_i)$ has the $\ybox$ at the end of the $i$-th lowest row of $\gamma^c/\alpha$. Let $X \subset \gamma^c/\alpha$ be the subtableau consisting only of the entries in the $(i+2)$-th row and above. We analyze the rectification ${\bf R} = \sh_2\sh_1(\sh_2{\bf D}_i)$, using the highest-weight tableau $T$ of shape $\alpha$. Note that the $\ybox$ must end up as the first entry in the second row in ${\bf R}$, and that $\sh_1({\bf R})$ slides the $\ybox$ upwards. We claim the following: no square of the rectification path of the $\ybox$ is immediately south or east of any square on the rectification path of any box of $X$. So, when computing ${\bf D}_j := \sh_1\sh_2(\sh_1{\bf R})$, the rectified squares of $X$ must return to their original locations. It follows that $j \leq i+1$.

Let $a_j$ be the number of boxes in the $j$-th lowest row of $\gamma^c/\alpha$. To prove the claim, we observe the following: when we compute $\sh_1(\sh_2 {\bf D}_i)$, the squares in the $j$-th lowest row of $\gamma^c/\alpha$ first slide left until the leftmost is in column $a_1 + \cdots + a_{j-1}+1$, then directly upwards. In particular, the leftmost box of $X$ lands in column $c = a_1 + \cdots + (a_i - 1) + a_{i+1} + 1$. Similarly, when we compute $\sh_2\sh_1(\sh_2{\bf D}_i)$, the $\ybox$ slides left to column $c' = a_1 + \cdots + a_i$, then up to row 2, then left to column 1. (See Figure \ref{fig:pieri-i}.) For the first set of slides, $\ybox$ is at least two rows lower than any square of $X$; afterwards it is strictly left of any square of $X$, since $c' < c$.
\end{proof}

\begin{figure}[h]
\centering
\includegraphics[scale=0.7]{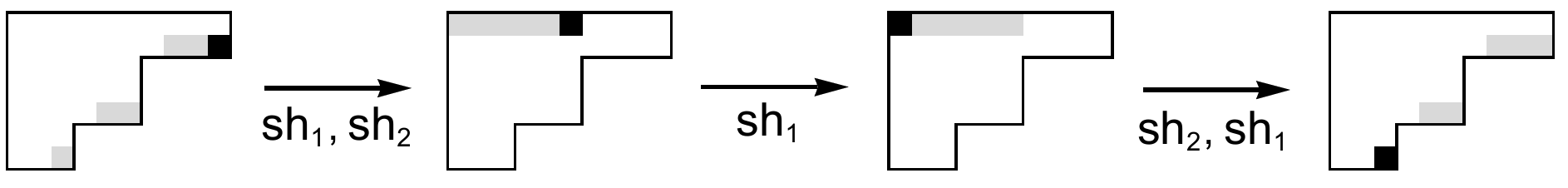}
\caption{Applying $\esh_2$ to the chain $\sh_2({\bf D_r})$. The result is ${\bf D}_1$.}
\label{fig:pieri-r1}
\end{figure}

\begin{figure}[h]
\centering
\includegraphics[scale=0.5]{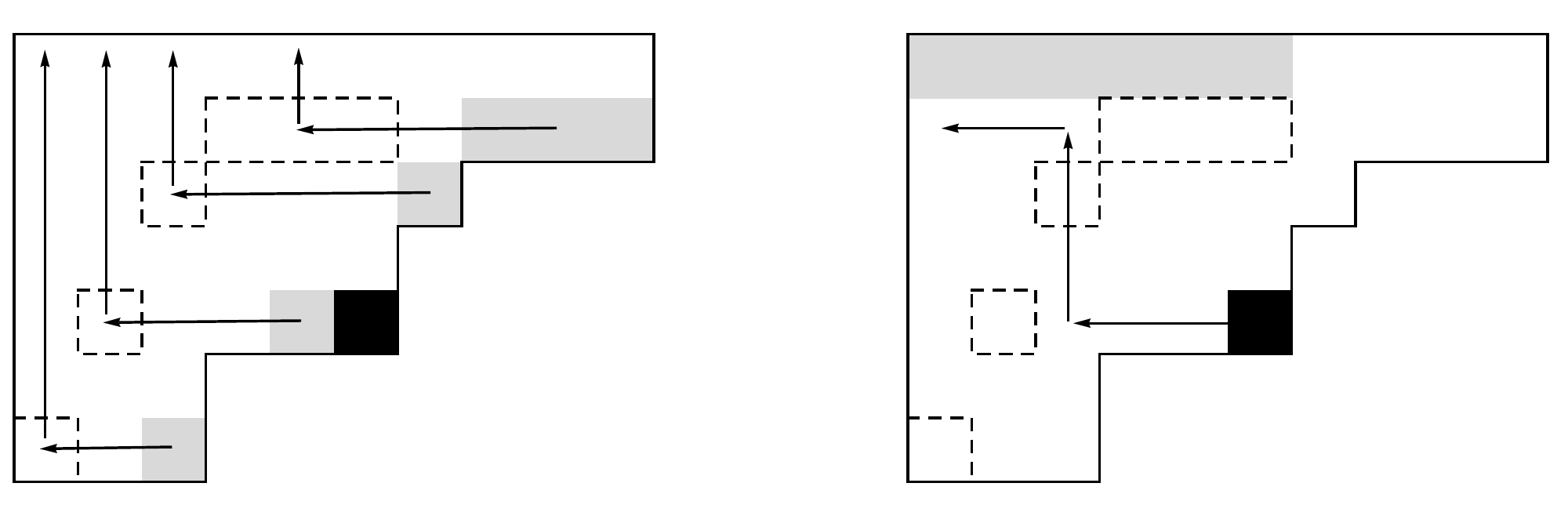}
\caption{Showing that $\omega(2) \leq 3$. The rectification path taken by the black box is never immediately south or east of the path taken by the highest strip.}
\label{fig:pieri-i}
\end{figure}

Finally, we recall the following description of $k_{\alpha \beta}^{\gamma^c}$:

\begin{prop}[\cite{ThYo}]
Let $\gamma^c/\alpha$ be a horizontal strip of size $d+1$ and $\beta = (d)$. Let $r$ be the number of nonempty rows in $\gamma^c / \alpha$. Then $k_{\alpha \beta}^{\gamma^c} = r-1$. The corresponding increasing skew tableaux are $K_{\alpha \beta}^{\gamma^c} = \{T_{12}, \ldots, T_{r,r-1}\}$, where the entries of $T_{i,i+1}$ are strictly increasing from left to right, except that the last entry of the $i$-th lowest row equals the first entry of the row above it.
\end{prop}

\begin{exa} Let $\gamma^c/\alpha = {\tiny \young(::::\hfil\hfil,::\hfil\hfil,\hfil)}$ and let $\beta = {\tiny \yng(4)}$. The corresponding tableaux are
\[X_\eset^\rect(\alpha, \ybox, \beta, \gamma) = \{{\bf D}_1, {\bf D}_2, {\bf D}_3\} = \bigg\{{\tiny \young(::::34,::12,\hfil), \young(::::34,::\hfil2,1), \young(::::\hfil4,::23,1)} \bigg\},\]
\[K_{\alpha \beta}^{\gamma^c} = \{T_{12}, T_{23}\} = \bigg\{{\tiny \young(::::34,::12,1), \young(::::34,::23,1)} \bigg\}.\]
We think of the tableau $T_{i,i+1}$ in $K$-theory as corresponding to the equation $\omega({\bf D}_i) = {\bf D}_{i+1}$ (for $i < r$).
\end{exa}

\begin{cor}
If $\gamma^c/\alpha$ is a horizontal strip of size $d+1$ and $\beta = (d)$, then $\omega$ has one orbit on $X_{\eset}^{\rect}(\alpha, \ybox, \beta, \gamma)$. In particular, the corresponding Schubert curve is irreducible, hence isomorphic to $\PP^1$ (since $\chi(\cO_S) = k_{\alpha \beta \tinybox \gamma}^{\rect} - k_{\alpha \beta}^{\gamma^c} = r - (r-1) = 1.$)
\end{cor}

\subsection{The promotion case} We consider the case of $N = k(n-k)-1$ copies of $\ybox$. In this case $\omega$ is given by tableau promotion on
$X_\eset^{\rect}(\ybox, \ldots, \ybox) = SYT(\rect)$, that is,
\[\omega = \sh_N \circ \cdots \circ \sh_1.\]
Note that, under the identification with $SYT(\rect)$, the evacuation-shuffles correspond to the identity map. We claim the following:
\begin{prop}
We have $\displaystyle{\mathrm{sign}(\omega) = \sum_i \mathrm{sign}(\sh_i) = k_{\underbrace{\tinybox, \ldots, \tinybox}_N}^{\rect}. \ (\mathrm{mod}\ 2).}$
\end{prop}
\begin{proof}
Let $S \in SYT(\rect)$. Then $\sh_i$ is the $i$-th Bender-Knuth involution, so $\sh_i$ acts by swapping the $i$-th and $(i+1)$-th entries of $S$ if they are nonadjacent. Let $Y_i$ be the set of (unordered) pairs $\{S, S'\}$ of standard tableaux exchanged by $\sh_i$. Note that $\mathrm{sign}(\sh_i) = Y_i \ (\mathrm{mod}\ 2).$

In $K$-theory, it follows from the K-theoretic Pieri rule that $k_{(\tinybox^N)}^{\rect}$ is the number of increasing tableaux $T$ of shape $\rect$ with entries $1, \ldots, N$. In particular, any such $T$ has a single repeated entry $i$, which occurs exactly twice in nonadjacent boxes. Let $X_i$ be the set of tableaux for which the repeated entry is $i$. Given $T \in X_i$, let $T'$ be the tableau in which the $i$'s are replaced by $*$ and each entry $j > i$ is replaced by $j+1$:
\[T = \young(123,345) \ \  \longrightarrow \ \  T' = \young(12*,*56).\]
We may set either $*$ of $T'$ to be $i$ or $i+1$, and so obtain a pair of \emph{standard} tableaux $S, S'$, and it is clear that $\sh_i(S) = S'$. This gives a bijection $X_i \to Y_i$.
\end{proof}

%
%
\appendix

\section{Simple nodes}

We will need the following standard lemma on simple nodes:

\begin{lemma}[Universal property of simple nodes] \label{nodal-appendix} Let $X,Y$ be $S$-schemes and let $\sigma_X : S \to X,\ \sigma_Y : S \to Y$ be $S$-points. There exists a scheme $\displaystyle{Z = X \bigsqcup_{\sigma_X \sim \sigma_Y} Y}$ over $S$, unique up to unique isomorphism, called the nodal gluing of $X$ to $Y$ along $\sigma_X$ and $\sigma_Y$, with the following properties:
\begin{enumerate}
\item[(i)] There are closed embeddings $i_X : X \hookrightarrow Z$ and $i_Y : Y \hookrightarrow Z$ over $S$, such that (identifying $X$ and $Y$ with their images in $Z$) we have $X \cup Y = Z$ and the scheme-theoretic intersection $X \cap Y$ is $\sigma_X(S)$ in $X$ and $\sigma_Y(S)$ in $Y$.
\item[(ii)] $Z$ is the pushout of the diagram of inclusions
\[\xymatrix{
& X \ar@{-->}[dr]^{i_X} \\
S \ar^{\sigma_X}[ur] \ar_{\sigma_Y}[dr] & & Z \\
& Y \ar@{-->}[ur]_{i_Y}
}\]
of $S$-schemes. In particular, a morphism $Z \to Z'$ over $S$ is the same as a pair of morphisms $X \to Z', Y \to Z'$ over $S$ that agree along $\sigma_X$ and $\sigma_Y$.
\end{enumerate}
Moreover, the first property implies the second, hence also characterizes $Z$.
\end{lemma}
\begin{rmk}
We will only need this lemma in the case where $S = \Spec(k)$, a point, though the proof is the same either way. In this case (i) says that whenever $Z$ is the union of two subschemes $X,Y$ whose scheme-theoretic intersection is one reduced $k$-point, $Z$ is the nodal gluing of $X$ to $Y$.
\end{rmk}
\begin{proof}
To show that $Z$ exists, we may take $Z$ to be the union of the fibers $X \times_S \sigma_Y(S) \cong X$ and $\sigma_X(S) \times_S Y \cong Y$ in $X \times_S Y$. This clearly satisfies (i).

To show that (i) implies (ii), let $\sI_X,\sI_Y$ be the ideals of $X$ and $Y$ in $Z$. Let $\sigma : S \to Z$ be the induced $S$-point, and let $\sI_S$ be the ideal of $\sigma(S)$ in $Z$. We have the short exact sequence
\[0 \to \cO_Z /(\sI_X \cap \sI_Y) \to \cO_Z/\sI_X \oplus \cO_Z/\sI_Y \to \cO_Z/(\sI_X+\sI_Y) \to 0,\]
where the first map is the inclusion and the second is subtraction. From (i), $\sI_X + \sI_Y = \sI_S$ and $\sI_X \cap \sI_Y = 0$, so we in fact have
\[0 \to \cO_Z \to \cO_X \oplus \cO_Y \to \cO_Z/\sI_S \cong \cO_S \to 0.\]
Suppose we have morphisms $X \to Z', Y \to Z'$ over $S$ that agree along $\sigma_X$ and $\sigma_Y$.  We must glue these morphisms together. We have maps $\cO_{Z'} \to \cO_X$ and $\cO_{Z'} \to \cO_Y$, fitting in the diagram
\[\xymatrix{
\cO_{Z'} \ar@/^/[drr] \ar@/_/[ddr] \ar@{-->}[dr]\\
& \cO_Z \ar[r] \ar[d] & \cO_X \ar[d]^{\sigma_X^\#} \\
& \cO_Y \ar[r]_{\sigma_Y^\#} & \cO_S
}\]
and we wish to construct the map $\cO_{Z'} \to \cO_Z$. In particular, let $f \in \cO_{Z'}$ be a section with images $f|_X \in \cO_X, f|_Y \in \cO_Y$. Mapping to $\cO_S$, we have $\sigma_Y^\#(f|_Y) = \sigma_X^\#(f|_X)$, so by the short exact sequence, there exists a unique $g \in \cO_Z$ mapping to $f|_X$ and $f|_Y$. This gives the desired map.
\end{proof}
\begin{cor} \label{nodal-iso-appendix}
Let $X,Y,Z,S$ be as above, and let $f : A \to Z$ be a morphism of schemes over $S$ whose restrictions $f^{-1}(X) \to X$ and $f^{-1}(Y) \to Y$ are isomorphisms. Then $f$ is an isomorphism.
\end{cor}
\begin{proof}
This follows from using the universal property to construct the inverse map $Z \to A$.
\end{proof}

\bibliographystyle{alpha}
\bibliography{schubert}

\end{document}